%% file: main.tex
\documentclass[12pt,reqno]{amsart}
\usepackage{graphicx,indentfirst}
\usepackage{amsmath,amssymb,mathrsfs}
\usepackage{amsthm,amscd}
\usepackage{verbatim}
\usepackage{appendix}
\usepackage{enumitem,titletoc}

\usepackage[utf8]{inputenc}
\usepackage{fancyhdr}
\usepackage{amsfonts,color}
\usepackage[all]{xy}
\usepackage{tikz-cd}
\usepackage{syntonly}
\usepackage{float}

\usepackage{array}
\usepackage[left=2.5cm,right=2.5cm,bottom=3cm,top=3cm]{geometry}

\usepackage{tikz}
\usepackage{extarrows}
\usepackage{hyperref}
\usepackage{cancel}

\def\XXint#1#2#3{{\setbox0=\hbox{$#1{#2#3}{\int}$ }
		\vcenter{\hbox{$#2#3$ }}\kern-.6\wd0}}

\newtheorem{theorem}{Theorem}[section]
\newtheorem{claim}{Claim}[section]
\newtheorem{lem}{Lemma}[section]
\newtheorem{prop}{Proposition}[section]
\newtheorem{cor}{Corollary}[section]
\newtheorem{ques}{Question}[section]

\newtheorem{remark}{Remark}[section]

\theoremstyle{definition}
\newtheorem{defn}{Definition}[section]

\newcommand{\del}{\partial}

\newcommand{\ddbar}{\sqrt{-1}\partial\bar\partial}

\makeindex

\definecolor{ForestGreen}{RGB}{34,139,34}
\hypersetup{
	colorlinks=true,
	citecolor=ForestGreen,
	filecolor=ForestGreen,
	linkcolor=blue,
	urlcolor=black
}

\numberwithin{equation}{section}

\author{Tristan C. Collins}
\email{\href{mailto:tristanc@math.toronto.edu}{tristanc@math.toronto.edu}}
\address{Department of Mathematics, University of Toronto, 40 St. George Street, Toronto, ON, Canada}

\author{Benjy Firester}
\email{\href{mailto:benjyfir@mit.edu}{benjyfir@mit.edu}}
\address{Department of Mathematics, Massachusetts Institute of Technology, 77 Massachusetts Ave., Cambridge, MA, USA}

\date{\today}
\input{derivatives}

\input{benjymacros}

\title[Free boundary Monge-Amp\`ere equations ]{On a general class of Free Boundary Monge-Amp\`ere equations}

\nc{\refHa}{$\mr{(}$\hyperref[def:H1]{$\mathbf{H1}$}$\mr{)}$}
\nc{\refHb}{$\mr{(}$\hyperref[defn: PropertyH2]{$\mathbf{H2}$}$\mr{)}$}
\nc{\refHc}{$\mr{(}$\hyperref[defn: PropertyH3]{$\mathbf{H3}$}$\mr{)}$}

\nc{\refsHa}{$\mr{(}$\hyperref[def:H1]{$\mathbf{sH1}$}$\mr{)}$}
\nc{\refsHb}{$\mr{(}$\hyperref[defn: PropertyH2]{$\mathbf{sH2}$}$\mr{)}$}
\nc{\refsHc}{$\mr{(}$\hyperref[defn: PropertyH3]{$\mathbf{sH3}$}$\mr{)}$}

\nc{\refHai}[1]{$\mr{(}$\hyperref[def:H1]{$\mathbf{H1}\mr{(#1)}$}$\mr{)}$ } 
\nc{\refHbi}[1]{$\mr{(}$\hyperref[def:H2]{$\mathbf{H2}\mr{(#1)}$}$\mr{)}$ }
\nc{\refHci}[1]{$\mr{(}$\hyperref[def:H3]{$\mathbf{H3}\mr{(#1)}$}$\mr{)}$ }

\nc{\hpropAa}{$\mr{(}$\hyperref[hpropA1]{$\mathbf{A1}$}$\mr{)}$}
\nc{\hpropAb}{$\mr{(}$\hyperref[hpropA2]{$\mathbf{A2}$}$\mr{)}$}

\begin{document}
\begin{abstract}
    We solve a general class of free boundary Monge-Amp\`ere equations given by 
    \[
    \begin{split}
    \det D^2u &=  \lambda \dfrac{f(-u)}{g(u^\star)h(\nabla u)}\chi_{\{u<0\}} \; \text{ in } \mathbb{R}^n,\\
    \nabla u (\mathbb{R}^n) &= P
    \end{split}
    \]
    where $P$ is a bounded convex set containing the origin, and $h>0$ on $P$. 
    We consider applications to optimal transport with degenerate densities, Monge-Amp\`ere eigenvalue problems, and geometric problems including a hemispherical Minkowski problem and free boundary K\"ahler-Ricci solitons on toric Fano manifolds.

\end{abstract}
\maketitle

\section{Introduction}
This paper studies a general class of free boundary Monge-Amp\`ere equations combining the classical first and second boundary value problems. 
Precisely, given a bounded convex domain $P \subset \bR^n$ containing the origin equipped with density $h(y)dy$, we establish the existence of solutions to Monge-Amp\`ere equations of the type
\begin{equation}\label{eqn:FreeBdyEqn}
    \begin{split}
        \det D^2u &= \lambda  \dfrac{f(-u)}{g(u^\star)h(\nabla u)}  \; \text{ in }\Omega  \text{ for some } \lambda \in \mathbb{R}_{+},\\
        u\big\vert_{\p \Omega} &= 0 ,\\
        \nabla u(\Omega) &= P,
    \end{split}
\end{equation}
where $u^\star =u^*(\nabla u) = \la x, \nabla u(x) \rg-u(x)$ is the Legendre transform of $u$ pulled back by the gradient map of $u$. Equation~\eqref{eqn:FreeBdyEqn} can be equivalently formulated in terms of the Legendre transform $v=u^*$ as
\begin{equation}\label{eqn:FreeBdyEqnLegendre}
    \begin{split}
        \det D^2v &= \lambda^{-1} \dfrac{g(v)h(y)}{f(-v^\star)} \;  \text{ in }P \text{ for some } \lambda \in \mathbb{R}_{+},\\
        v^\star\big\vert_{\p P} &= 0 
    \end{split}
\end{equation}
where, as before $v^\star(x) = \la x,\nabla v\rg - v$ is $v^*=u$ pulled back by $\nabla v$. If $\Omega \subset \mathbb{R}^n$ is fixed, then equation~\eqref{eqn:FreeBdyEqn} is overdetermined, so we must allow $\Omega = \{u <0\}$ to be undetermined, and hence $\del\Omega$ is a free boundary. 

Our main results establish the existence of solutions to equations~\eqref{eqn:FreeBdyEqn} and~\eqref{eqn:FreeBdyEqnLegendre} under general structural assumptions on the functions $f,g, h$. 
We delay stating the precise assumptions we impose until the following subsection, choosing instead to first focus on some applications of our results. 

As a first application, we obtain the following existence result for homogeneous optimal transport maps between a proper cone and a half-space, both equipped with homogeneous, degenerate densities:

\begin{theorem}\label{thm: introOT}
Let $P\subset \mathbb{R}^n$ be a bounded convex set, and let $\mathtt{C}(P)= \{ (y_{n+1}y',y_{n+1})\in \mathbb{R}^{n+1} : y' \in P, y_{n+1}\in \mathbb{R}_{>0}\}$ be the affine cone over $P$.
Let $h(y') \sim d(y',\del P)^{\alpha}$ for some $\alpha\geq0$, and let $\rho^{\alpha}$ denote the degree $\alpha$ homogeneous extension of $h$ to $\mathtt{C}(P)$. 
Then, for any $\beta>\alpha$ there exists a convex function $\varphi: \mathtt{C}(P) \rightarrow \mathbb{R}_{>0}$, homogeneous of degree $1+\frac{n+1+\alpha}{n+1+\beta}$, such that
\[
\begin{aligned}
(\varphi_{n+1})^{\beta} \det D^2\varphi &= \rho^{\alpha},\\
\nabla \varphi (\mathtt{C}(P)) &= \mathtt{C}_+:=\{(x_1,\ldots, x_{n+1}) \in \mathbb{R}^{n+1} : x_{n+1}>0\}.
\end{aligned}
\]
That is, $\varphi$ defines a homogeneous optimal transport map from $(\mathtt{C}(P), \rho^{\alpha}(y)\,dy)$ to $(\mathtt{C}_+, x_{n+1}^{\beta}\, dx)$.
\end{theorem}

In the case of the uniform measure on the cone $\mathtt{C}(P)$, (i.e., $h\equiv1$ and $\alpha =0$), Theorem~\ref{thm: introOT} was obtained in dimension $n=1$ by Collins-Li\cite{Collins-Li} using ODE methods, and in general dimensions by Collins-Tong-Yau \cite{TristanFreidYau} by dimensionally reducing to the free boundary equation~\eqref{eqn:FreeBdyEqnLegendre} with $h(y)\equiv1$, $g(s) = s^{-(n+2)}$, and $f(s)= s^\beta$.
As discussed in \cite{TristanFreidYau, Collins-Li} these problems are motivated by the construction of complete Calabi-Yau metrics on certain log Calabi-Yau manifolds. 

Theorem~\ref{thm: introOT} is related to the regularity theory of optimal transport maps between convex domains with degenerate densities, following the recent work of the first author and Tong \cite{Collins-Tong}. See in particular \cite[Theorem 6.5]{Collins-Tong}; it is interesting to note that the condition $\beta>\alpha$ also appears there. 

A second application of our results is the following (see Theorem~\ref{thm: MAEV}):
\begin{theorem}\label{thm: Eigenvalue}
Let $P\subset \mathbb{R}^n$ be a bounded convex set with $0\in P$, and $k \in \mathbb{R}_{\geq 0}$. 
Then, there exists convex function $u: \mathbb{R}^n \rightarrow \mathbb{R}$, unique up to translations, such that
\[
\begin{aligned}
\det D^2u &= (-u)^k \chi_{\{u<0\}},\\
\nabla u(\mathbb{R}^n) &= P
\end{aligned}
\]
if and only if $0$ is the barycenter of $P$.
\end{theorem}

A consequence of this result is a ``reconstruction"-type theorem for the Monge-Amp\`ere equation. 
For example, given a bounded convex set $\Omega$, consider the classical Dirichlet problem for the Monge-Amp\`ere equation
\[
\begin{aligned}
    \det D^2u &=1 \; \text{ in } \Omega,\\
    u\big|_{\del\Omega} &=0.
    \end{aligned}
\]
Theorem~\ref{thm: Eigenvalue} implies that, up to translations, both $\Omega$ and $u$ can be reconstructed from $P=\nabla u(\Omega)$ when $P$ is convex; see Corollary~\ref{cor: reconstruction} for a precise (and more general) statement.

Finally, we mention a geometric application. Regard $\mathbb{R}^n \subset \mathbb{R}^{n+1}$ as the set $\{y_{n+1}=1\}\subset\mathbb{R}^{n+1}$, $\mathbb{R}^{n+1}_{+}= \{y_{n+1} \geq 0\}$, and let $\bS^{n}_+ = \bS^n \cap \mathbb{R}^{n+1}_{+}$ be the upper hemisphere. Let $\mathtt{s}:\bS^n_+\rightarrow \mathbb{R}^n$ be defined by 
\[
\mathtt{s}(y',y_{n+1}) = \frac{y'}{y_{n+1}}.
\]
Then, we have the following result:

\begin{theorem}\label{thm: IntroGaussCurvature}
Fix a bounded convex set $P\subset \mathbb{R}^n$, and suppose that
\[
\int_{P}\vec{y}(1+|y|^2)^{\frac{n+2}{2}}\, dy=0.
\]
Then, there exists a hemisphere $\Sigma \subset \mathbb{R}^{n+1}_{+}$, with $\del \Sigma \subset \{x_{n+1}=0\}$ having constant Gauss curvature, and such that the image of the Gauss map $\nu_{\Sigma}:\Sigma \rightarrow \bS^n$ is given by $\mathtt{s}^{-1}(P)$.
\end{theorem}

Theorem~\ref{thm: IntroGaussCurvature} is a special case of Theorem~\ref{thm: prescribedGauss}, which produces hemispheres of prescribed Gauss curvature with boundary on a hyperplane. This result can be viewed as a natural hemispherical version of the Minkowski problem.

\subsection{An outline of the argument, and technical assumptions}

Before stating the precise assumptions on the data $f,g,h$, let us give a brief overview of our approach, which is variational in nature. Fix antiderivatives 
\begin{equation}\label{eqn: introAntiDeriv}
F(s) = \int_0^s f(t)\, dt \quad \text{ and } \quad  G(s) = \int {g(t)}\,dt.
\end{equation}
Consider the class of convex functions on $P$ given by
\begin{equation}\label{eqn:defnC+}
\cC^+ = \{v: \ol{P} \to \bR_+, v \text{ convex}\, \},
\end{equation}
and consider the energy functional
\begin{equation}\label{eqn:IntroEnergy}
v \ni \cC^+ \mapsto \mathcal{E}(v)=-\log\left(\int_{\mathbb{R}^n}F(-v^*(x))\chi_{\{v^*<0\}}\,dx\right) + \Lambda G^{-1}\left(\frac{1}{H}\int_{P}G(v(y))\,h(y)dy \right)
\end{equation}
where $h(y)>0$ is a positive density on $P$ and $H= \int_{P}h(y)\,dy$. One can check formally that that critical points of $\mathcal{E}$ solve~\eqref{eqn:FreeBdyEqn}; see Section~\ref{sec2:Functionals}. 

When $g$ is increasing, then $\mathcal{E}$ is a convex and one can approach the existence of critical points by a minimization argument.
In contrast, when $g$ is positive but decreasing, so that $G$ is increasing and concave, the energy functional $\mathcal{E}$ is not convex, and is, in general, unbounded from above and below.
In this setting, adopting the approach of \cite{TristanFreidYau}, we establish existence of critical points via a min-max argument.
In both the convex or concave cases, structural conditions on the data $f,g,h$ are needed to guarantee that the variational problem admits a suitable compactness principle. 

There are essentially two sources of non-compactness which arise when studying the energy functional~\eqref{eqn:IntroEnergy}, which we illustrate below. 

\subsection{Translational non-compactness}\label{subsec: translationsIntro} The first source of non-compactness arises from {\em translations of the free boundary}.
Precisely, suppose that $v\in \cC^+$, and let $u=v^*$ be the Legendre transform.
For any point $x_0\in \{u<0\}$, we can consider the function $v_{x_0}(y):= v(y)- \langle x,y\rangle$, whose Legendre transform is given by $u_{x_0}^*(x)=u^*(x+x_0)$.
Then,
\[
-\log\left(\int_{\mathbb{R}^n}F(-v_{x_0}^*(x))\chi_{\{v_{x_0}^*<0\}}\,dx\right)=-\log\left(\int_{\mathbb{R}^n}F(-v^*(x))\chi_{\{v^*<0\}}\,dx\right)
\]
is invariant. 
On the other hand, 
\[
G^{-1}\left(\frac{1}{H}\int_{P}G(v_{x_0}(y))\,h(y)dy \right)= G^{-1}\left(\frac{1}{H}\int_{P}G(v(y)-\langle x_0,y\rangle)\,h(y)dy \right)
\]
is not invariant, and so we must require that the convex or concave function
\begin{equation}\label{eqn:transVariationIntro}
\{v^*<0\} \ni x \mapsto G^{-1}\left(\frac{1}{H}\int_{P}G(v(y)-\langle x,y\rangle)\,h(y)dy \right),
\end{equation}
has an interior critical point in $\{u<0\}$. For example, if $G(s)=s$, then this is only the case when,
\[
\int_{P}\vec{y}h(y)\,dy=0,
\]
and this condition is necessary for the existence of a solution to the associated free boundary Monge-Amp\`ere equation. Apart from this scenario, the only other condition we have found which guarantees that~\eqref{eqn:transVariationIntro} has an interior critical point is to require that $G(s)$ blows up to $-\infty$ sufficiently fast as $s\rightarrow 0$. 
For example, if $h >0$ on $\overline{P}$, we require that
\[
\int_{0}^{1} G(s)s^{n-1}\,ds =-\infty.
\]
One can then check that, as $x\rightarrow \del \{u<0\}$, the non-negative function~\eqref{eqn:transVariationIntro} goes to zero, and hence the maximum is necessarily achieved in the interior of $P$.

\subsection{Scaling and shifting non-compactness} The second source of non-compactness arises from $\inf_{P}v \rightarrow 0$, $\inf_{P}v \rightarrow +\infty$, and/or, $\sup_{P}|\nabla v| \rightarrow \infty$.
These behaviors are captured by the following family of ``enemy" functions $v_{\delta,C}$ parametrized by $(\delta, C) \in \mathbb{R}^{2}_{>0}$ satisfying $\delta \lesssim C$. 
For convenience, we specify $v_{\delta,C}$ by their Legendre transform, as follows:
\[
v^{*}_{\delta,C} = \max\{ \phi_{P}(x)-C, -\delta\} \quad \text{ where }\quad  \phi_{P}(x) = \sup_{y\in P} \la x, y\rg.
\]
For this family, the free boundary $\{v^*_{\delta,C} <0\} = CP^\circ$, where $P^{\circ}$ is the polar dual of $P$. 
One can then estimate 
\[
\begin{aligned}
\mathcal{E}(v_{\delta,C})  &\sim  -\log(F(\delta)C^n) +\Lambda G^{-1}\left(\int_{0}^{1}G(\delta+tC)t^{n-1} \,dt\right).
\end{aligned}
\]
To rule out non-compactness for this particular family, one needs conditions which imply the set
\[
\{(\delta, C) : \mathcal{E}(v_{\delta,C}) \leq A\} \Subset \mathbb{R}^{2}_{>0}
\]
is compact. The preceding example is, of course, overly simplistic; in practice, the free boundary need not become large in all directions uniformly.

\subsection{Structural properties of \texorpdfstring{$(F,G)$}{(F,G)}}

We now discuss the structural properties of $(F,G)$ required to control the non-compactness of our variational problem. We will use the following notation: we say that $a\lesssim b$ if there is a uniform constant $C$, such that $a \leq Cb$. We say $a\sim b$ if $a\lesssim b$ and $b \lesssim a$. We refer the reader to Section~\ref{sec: notation} for a complete list of the notation used in the paper.

The first case is the ``convex case" when $G(s)=s$. 
\begin{defn}[Property $\mathbf{H1}$]\label{def:H1}
    We say that a pair $(F,G)$ satisfies Property~\refHa\ if
    \begin{itemize}
        \item[(1)] $G'(s)= \text{const.}>0$.
        \item[(2)] $\log(F(t)) = O(t)$ as $t \rightarrow \infty$.
    \end{itemize}
    We say that $(F,G)$ satisfies {\bf Strong} Property~\refHa, denoted \refsHa, if \refHa\ holds, together with
    \begin{itemize}
        \item[(2')] $\log(F(t)) = o(t)$ as $t \rightarrow \infty$.
    \end{itemize}
\end{defn}

In the concave case, we consider two different structural conditions depending on whether the ``dominant" behavior of $G(s)$ occurs for large or small values of $s$.

\begin{defn}[Property $\mathbf{H2}$]\label{defn: PropertyH2}
    We say that a pair $(F,G)$ satisfies Property \refHb${}_{\nu,\gamma}$ for $(\nu, \gamma) \in \mathbb{R}^2_{\geq 0} $ if
    \begin{enumerate}
        \item $G: (0,\infty) \rightarrow \mathbb{R}$ is concave, increasing, $G(1)=0$, and $\lim_{s\rightarrow \infty}G(s)=+\infty$.
        \item For $ s\leq 1$, we have
        \[
        -s^{-(n+\nu)} \lesssim G(s) \lesssim -s^{-(n+\nu)+1} \quad \text{ and }\quad  \int_{0}^{1}G(s)s^{n+\nu-1}\,ds = -\infty.
        \]
         \item $F(s):[0,\infty) \rightarrow \mathbb{R}$ is convex, increasing, $F(0)=0$, and for $s\leq 1$, $F(s)\sim s^{\gamma}$ for some $\gamma \geq \nu \geq 0$.
         \item For all $\epsilon \in (0,1)$, we have
         \[
         \log(F(x)) =  O(G^{-1}(\epsilon G(\epsilon x)) \quad \text{ as } x \rightarrow +\infty.
         \]
    \end{enumerate}
    We say that a pair $(F,G)$ satisfies {\bf Strong} Property \refHb${}_{\nu,\gamma}$, denoted \refsHb${}_{\nu,\gamma}$, if, in addition, $(F,G)$ satisfy
    \begin{itemize}
        \item[(4')] For all $\epsilon \in (0,1)$, we have
         \[
         \log(F(x)) =  o(G^{-1}(\epsilon G(\epsilon x)) \quad \text{ as } x \rightarrow +\infty.
         \]
    \end{itemize}
    \end{defn}

The next definition concerns the case when the dominant behavior of $G(s)$ occurs for small values of $s$.

\begin{defn}[Property $\mathbf{H3}$]\label{defn: PropertyH3}
    We say that a pair $(F,G)$ satisfies Property \refHc${}_{\nu,\gamma, \vec{\beta}}$\ for $(\nu, \gamma) \in \mathbb{R}^2_{\geq 0} $ and $\vec{\beta}= (\beta_1, \beta_2) \in \mathbb{R}^2_{\geq 0}$ if
    \begin{enumerate}
        \item $G:(0,\infty) \rightarrow \mathbb{R}_{<0}$ is concave, increasing and $\lim_{s\rightarrow \infty} G(s) =0$.
        \item For $ s\leq 1$, we have
         \[
        -s^{-(n+\nu)} \lesssim G(s) \lesssim -s^{-(n+\nu)+1} \quad\text{ and }\quad  \int_{0}^{1}G(s)s^{n+\nu-1}\,ds = -\infty.
        \]
        \item For $s \geq 1$, we have $\beta_2 \geq \beta_1 \geq 0$ such that
        \[
        -s^{-(n+\beta_1)} \lesssim G(s) \lesssim -s^{-(n+\beta_2)} \quad \text{ and } \quad \bigg|\int_{1}^{\infty}G(s)s^{n-1}\,ds\bigg| < +\infty.
        \]
        \item $F(s):[0,\infty) \rightarrow \mathbb{R}$ is convex, increasing, $F(0)=0$, and for $s \leq 1$,  $F(s)\sim s^{\gamma}$ for some $\gamma > \nu \geq 0$.
        \item $\log(F(s)) = O\big(s^{\frac{n+\beta_1}{n+\beta_2}}\big)$ as $s \rightarrow +\infty$. 
    \end{enumerate}
    We say that a pair $(F,G)$ satisfies {\bf Strong} Property \refHc${}_{\nu,\gamma, \vec{\beta}}$, denoted \refsHc${}_{\nu,\gamma, \vec{\beta}}$, if, in addition, $F$ satisfies
    \begin{enumerate}
\item[(5')]$\log(F(s)) = o\big(s^{\frac{n+\beta_1}{n+\beta_2}}\big)  \text{ as } s \rightarrow +\infty. $
    \end{enumerate}
    \end{defn}

\begin{remark}
    For convenience, we will often suppress the dependence of Properties \refHb${}_{\nu,\gamma}$ and \refHc${}_{\nu,\gamma,\vec{\beta}}$\ on their respective parameters, referring to them as \refHb\ and \refHc, for simplicity.
\end{remark}

We will only distinguish between properties \refHb\ and \refHc\ and their stronger counterparts \refsHb\ and \refsHc\ in the compactness result, Proposition~\ref{prop:EEhatcomparisonConcaveb}.

In Properties~\refHb, and~\refHc, condition (2) is used to control the translational non-compactness. We note that the upper bound in condition (2) follows automatically from concavity and the assumption that $\int_{0}^{1}G(s)s^{n+\nu-1}\, ds = -\infty$. Similarly, the lower bound in condition (3) of Property \refHc\ follows automatically from the integrability assumption.

Despite their technical appearance, these conditions are rather flexible, and are typically easy to check in practice. For example:
\begin{itemize}
\item Property \refHa, holds for any convex function $F(s)$ satisfying the growth constraint $F(s) \leq e^{Cs}$ for some $C>0$ as $s\rightarrow \infty$. 
\item Property~\refHb, holds for any concave function $G(s)$ equal to $-s^{-n}$ or $\frac{s^{-n}}{\log(s)}$ for $s  \ll 1$, and equal to $\log(s)$ for $s \gg 1$, provided $F$ has polynomial growth at infinity.
\item Property~\refHc\ holds for $F(s)= s^{k+\delta}$ for any $k\geq 1$, and $\delta >0$, and $G(t)=s^{-(n+k)}$.
\end{itemize}

Further examples and applications will be given in Section~\ref{sec5:ExamplesApplications}.

\subsection{Structural conditions on \texorpdfstring{$h(y)$}{h(y)}}\label{sec: hStructure}

We now explain the assumptions we make on the measure $h(y)dy$. 

\begin{defn}[Vanishing order]\label{defn: vanishingOrder}
    We say that $\mathtt{v}_o(h) \in [0,\infty)$ is a vanishing order of $h$ if there exists a cone $\mathtt{C}$, with apex at the origin, and constants $C, \delta >0$ with the following property:
    \begin{itemize}
    \item[(i)] For all $y_0 \in \del P$, there exists $A \in O(n)$, such that
    \[
    K_{y_0} := \{y_0 + A\cdot \mathtt{C} \} \cap B_{\delta}(y_0) \subset P.
    \]
    \item[(ii)] For all $y \in K_{y_0}$, we have the bound
    \[
    h \geq C^{-1}d(y,y_0)^{\mathtt{v}_o}.
    \]

   \end{itemize}
  
\end{defn}

\begin{remark}\label{rk: vanishOrder}
    We make two comments on the vanishing order. First, note if $h >0$ on $\overline{P}$, then $\mathtt{v}_0(h)=0$ is a vanishing order. Secondly, if $h(y) \sim d(y,\del P)^{\alpha}$, then $\mathtt{v}_0(h)=\alpha$ is a vanishing order of $h$, and the cone $C$ may be chosen depending only on $P$.
\end{remark}

\begin{defn}[Doubling measure]\label{defn: doublingMeasure}
We say that $dy^h := h(y)dy$ defines a doubling measure on $P$ if there exists a constant $C>0$ such that, for every point $y_0 \in \overline{P}$, and every ellipsoid $E$ centered at the origin, we have
\[
\int_{y_0+E} dy^h \leq C \int_{y_0+\frac{1}{2}E} dy^h,
\]
where we extend $h$ by zero outside of $\overline{P}$.
\end{defn}

\begin{defn}[$h$-barycenter]\label{defn: hbarycenter}
We say that $P$ has $h$-barycenter $y_0\in P$ if
\[
\int_{P}(\vec{y}-\vec{y}_0)\, dy^h = 0.
\]
\end{defn}

We now summarize the assumptions we make on $h(y)dy$.
\begin{enumerate}
\item[($\mathbf{A1}$)] $h: \ol{P}\rightarrow \mathbb{R}_{\geq 0}$, $h(y)>0$ on $P$, and $h(y)dy$ defines a doubling measure on $P$.\phantomsection\label{hpropA1}
\item[($\mathbf{A2}$)] We assume that $h(y)dy$ has a well-defined vanishing order $\mathtt{v}_0(h) \in[0,+\infty)$.\phantomsection\label{hpropA2}
\end{enumerate}

\begin{remark}
    The main theorems continue to hold under significantly weaker assumptions than \hpropAa.
    For example, it is not hard to check that our results continue to hold if $h$ vanishes in $P$ on subsets of codimension at least $1$, provided a growth condition analogous to that in Definition~\ref{defn: vanishingOrder} is satisfied.
\end{remark}

We note the following simple lemma, which will be useful later in the paper.

\begin{lem}\label{lem: poleNorm}
Suppose that $\mathtt{v}_o$ is a vanishing order of $h$ and $G$ satisfies Property \refHb${}_{\nu,\gamma}$\ or Property \refHc${}_{\nu,\gamma,\vec{\beta}}$\ for $\nu \geq \mathtt{v}_o$. 
Then, for all $y_0 \in \ol{P}$, $\delta > 0$, and $R > 0$, 
\begin{equation}\label{eqn:secondPoleCondition}
\int_{P\cap \{|y-y_0|<\delta\}} G\bigl(R |y-y_0|\bigr)\,dy^h = -\infty.
\end{equation}
\end{lem}
\begin{proof}
Fix a cone $\mathtt{C}$ and $A \in O(n)$ such that $K_{y_0}:= \{y_0+A\cdot \mathtt{C}\} \cap B_{\delta}(y_0) \subset P$, and such that, on $K_{y_0}$ we have $h(y) \gtrsim d(y,y_0)^{\mathtt{v}_0}$. Then we have
\[
G\bigl(R |y-y_0|\bigr)h(y) \lesssim G\bigl(R |y-y_0|\bigr)|y-y_0|^\mathtt{v_o}.
\]
But, since $\nu \geq \mathtt{v}_o$, the integral
\[
\int_{K_{y_0}}G\bigl(R |y-y_0|\bigr)|y-y_0|^\mathtt{v_o} dy
\]
diverges by assumption $(2)$ of Property \refHb${}_{\nu,\gamma}$\ or Property \refHc${}_{\nu,\gamma,\vec{\beta}}$\ .
\end{proof}

With these definitions, we can now state our main theorems. 
\begin{theorem}\label{thm:Barycenter}
Suppose that $(F,G)$ satisfy Property \refHa and $h(y)$ satisfies \hpropAa\ and \hpropAb. 
Then, there exists $\Lambda_0 \geq 0$, depending only on universal data, such that if $\Lambda > \Lambda_0$, then there exists a $v \in C^{1,\gamma}(\ol{P})$ for some $\gamma>0$ solving equation~\eqref{eqn:FreeBdyEqnLegendre} if and only if $P$ has $h$-barycenter at $0\in P$. 
In~\eqref{eqn:FreeBdyEqnLegendre},  $\lambda=\lambda(\Lambda)$ is determined implicitly by $\Lambda, v$ according to the formula
\[
\lambda(\Lambda) = \frac{g(G(J(v)) H}{\Lambda I(v)}, \quad \text{ where }  \qquad H= \int_P h(y)dy.
\]
If $(F,G)$ satisfy \refsHa, then we may take $\Lambda_0=0$.
\end{theorem}

In the concave cases, we have the following:

\begin{theorem}\label{thm: mainH23}
Suppose that $h(y)$ satisfies \hpropAa\ and \hpropAb, and $(F,G)$ satisfy Property \refHb${}_{\nu,\gamma}$, or Property \refHc${}_{\nu,\gamma,\vec{\beta}}$, for $\nu \geq \mathtt{v}_o(h)$, for some vanishing order of $h$. Then there exists a $\Lambda_0 \geq 0$, depending only on universal data, such that if $\Lambda > \Lambda_0$, then there exists a $v \in C^{1,\gamma}(\overline{P})$ solving equation~\eqref{eqn:FreeBdyEqnLegendre}. 
In~\eqref{eqn:FreeBdyEqnLegendre},  $\lambda=\lambda(\Lambda)$ is determined implicitly by $\Lambda, v$ according to the formula
\[
\lambda(\Lambda) = \frac{g(G(J(v)) H }{\Lambda I(v)}, \quad \text{ where }  \qquad H= \int_P h(y)dy.
\]
Furthermore, if $(F,G)$ satisfy \refsHb${}_{\nu, \gamma}$ or \refsHc${}_{\nu, \gamma, \vec{\beta}}$, then we may take $\Lambda_0=0$.
\end{theorem}

Setting $u = v^*$ and taking the convex envelope of $u$ solves the equation
\[
\begin{split}
    \det D^2 u &= \lambda \frac{f(-u)}{g(u^\star)h(\nabla u)}\chi_{\{u < 0\}} \; \text{ in }\bR^n\\
    \nabla u (\bR^n) &= P.
\end{split}
\]
If $f,g,h \in C^{k,\alpha}(P)$, then the solution $v \in C^{1,\alpha}(\ol{P}) \cap C^{k+2,\beta}(P)$ for $\beta <\alpha$ by the interior regularity theory for the Monge-Amp\`ere equation \cite{Caffarelli, Caffarelli2, Caffarelli3, Caffarelli4, Caffarelli5}. 

Our work builds on the ideas introduced by the first author, Tong, and Yau \cite{TristanFreidYau}. 
However, the techniques used to obtain the key compactness results in Section~\ref{sec3:Compactness} are completely different. Such modifications are necessary to address the generality of this problem. 
Indeed, the results in \cite{TristanFreidYau} used in an essential way the particular structure of the problem when $h=1, f(s)=s^k, g(s)=s^{-(n+2)}$, and in particular relied on some remarkable integration formulas discovered by Tong-Yau \cite{Tong-Yau}.

\subsection{Outline}
In Section~\ref{sec2:Functionals}, we introduce the variational approach to solving~\eqref{eqn:FreeBdyEqn}. Section~\ref{sec3:Compactness} is the technical heart of the paper, establishing estimates for the energy functions guaranteeing the compactness of the variational problem. In Section~\ref{sec4:Proof}, we complete the proofs of Theorems~\ref{thm:Barycenter} and~\ref{thm: mainH23}. Here we exploit results from optimal transport to show that the limit of a minimizing sequence solves the equation in a weak sense. Finally, in Section~\ref{sec5:ExamplesApplications}, we give several examples and applications of our results. We prove Theorem~\ref{thm: introOT}, Theorem~\ref{thm: Eigenvalue}, and solve a hemisphereical version of the Minkowski problem, which yields Theorem~\ref{thm: IntroGaussCurvature}. We also discuss some natural questions, and discuss which of our assumptions are necessary and which can possibly be weakened. 

\subsection{Notation and terminology}\label{sec: notation}
We briefly introduce the notation and terminology we use throughout the paper.
\begin{itemize}
    \item For a convex set $K$, $\phi_{K}$ will denote the convex support function of $K$, defined by $ \phi_{K}(x) = \sup_{y\in K} \langle x,y \rangle $.
    \item $P$ will denote an open, bounded convex set. $P^{\circ} = \{\phi_P <1\}$ will denote the polar dual of $P$.
    \item We will use the notation $a \lesssim b$ to mean that $a \leq Cb$ for some uniform, estimable constant $C$ depending only on fixed data. $a \sim b$ means $a \lesssim b$ and $b \lesssim a$.
    \item A {\em universal constant} is any constant which can be estimated in terms of properties of $n, P$ and the implicit constants in Definition~\ref{def:H1}, Definition~\ref{defn: PropertyH2} or Definition~\ref{defn: PropertyH3}. 
    \item We will denote by $dy^h$ the measure  $h(y)\, dy$, with $h$ satisfying assumptions \hpropAa\ and \hpropAb\ of Section~\ref{sec: hStructure}.
\end{itemize}

\medskip 

\noindent\textbf{Acknowledgments}: The authors are grateful to F. Tong for many helpful conversations. T.C.C. is supported in part by NSERC Discovery grant RGPIN-2024-518857, and NSF CAREER grant DMS-1944952. B.F. is supported by a MathWorks fellowship.
 
\section{Variational framework}\label{sec2:Functionals}
For $P$ a bounded convex set containing the origin, we define a class of convex functions $\cC^+$ as in Definition~\ref{eqn:defnC+}. For such $v\in \cC^+$, let $u$ be its Legendre transform $u(x) = \sup_{y \in P} (\la x,y\rg - v(y))$. 
The positivity of $v$ means $0<v(0) = -\inf_\Omega u$, which implies that $\Omega = \{u < 0\}$ is a bounded, convex set containing the origin. 
The function space $\cC^+$ is closed under convex combinations and subtracting linear functions $\la y, x_0\rg$ when $x_0 \in \Omega$.

Define the functionals
\[
I(u) = \int_{\bR^n} F(-u(x))\chi_{\{u< 0\}}\, dx = \int_P F(-v^\star(y))\det D^2 v\, dy
\]
and
\[
J(v) = G^{-1}\left(\frac{1}{H}\int_P G(v(y)) \, dy^h\right),
\]
where $dy^h = h(y)dy$ and $H = \int_P  \, dy^h$. Note that $J$ does not depend on the choice of anti-derivative $G$ of $g$. Finally, let
\[
\cE(v) = -\log I(u) + \Lambda J(v)
\]
where $\Lambda$ is some fixed, non-negative parameter. 
The motivation for this choice of energy functional is the following:
\begin{claim}
    If $v\in C^2(P)$ is a strictly convex critical point of $\cE$, then $v$ solves equation~\eqref{eqn:FreeBdyEqnLegendre}.
\end{claim}
\begin{proof}
The variation of $I$ is 
\[
\delta I(u) = -\int_{\bR^n} f(-u)\chi_{\{u < 0\}}(\delta u)\, dx.
\]
Equivalently, on $P$ let $\varphi$ be the variation, and
\begin{align*}
    \delta I(v) &= \int_P f(-v^\star(y))(\varphi - y\cdot \nabla \varphi)\det D^2v\, dy + \int_P F(-v^\star(y))\varphi_{ij}v^{ij}\det D^2v\, dy\\
    &= \int_P f(-v^\star(y))(\varphi - y \cdot \nabla \varphi)\det D^2v\, dy + \int_P f(-v^\star(y))(v_j - y_{k}v_{kj} - y_{kj}v_k)\varphi_iv^{ij}\det D^2v\, dy\\
    &=\int_P f(-v^\star(y))(\varphi)\det D^2v\, dy.
\end{align*}
The variation of $J$ is
\begin{equation}\label{eqn:variationOfJ}
\delta J(v) = \frac{1}{g(G(J(v)))}\frac{1}{H}\int_P g(v(y))(\varphi) \, dy^h.
\end{equation}
Therefore, the total variation of $\cE(v)$ is given by 
\[
\int_P \left[-\frac{1}{I(v)}f(-v^\star(y))\det D^2 v + \frac{\Lambda}{Hg(G(J(v)))} g(v)h(y)\right](\varphi)\, dy,
\]
so the Euler-Lagrange equation is
\[
f(-v^\star(y))\det D^2v = \frac{ \Lambda I(v)}{Hg(G(J(v)))}g(v)h(y),
\]
which is precisely equation~\eqref{eqn:FreeBdyEqnLegendre} up to choosing $\lambda$. 
Taking the Legendre transform rewrites this expression in terms of $u$ as
\[
\det D^2 u = \frac{Hg(G(J(u)))}{\Lambda I(u)}\frac{f(-u)}{g(u^\star)h(\nabla u)}\chi_{\{u< 0\}},
\]
which is equation~\eqref{eqn:FreeBdyEqn} provided $\frac{ Hg(G(J(u)))}{\Lambda I(u)} = \lambda$. 
\end{proof}

We now discuss some of the properties of the functionals $-\log(I)$ and $J$. 
The first is a simple bound on $I$, which follows immediately from the definition of the Legendre transform.
\begin{lem}\label{lem: IBound}
For any $v \in \cC^+$ we have the estimate 
\begin{equation}\label{eqn:IBound}
I(v) = \int_\Omega F(-u)\, dx \leq |\Omega|\inf_{x \in \Omega}F(-u(x)) = F(v(0))|\Omega|.
\end{equation}
\end{lem}
Recall that the space $\cC^+$ defined in~\eqref{eqn:defnC+} is convex in the sense that if $v_0,v_1 \in \cC^+$, then $v_t= (1-t)v_0+tv_1 \in \cC^+$ for $t\in[0,1]$.

\begin{claim}\label{claim:convexityOfI}
    Suppose $F$ is log-concave. 
    If $v_0, v_1 \in \cC^+$ and $v_t = (1-t)v_0+tv_1$, and $u_t = v_{t}^*$, then $t\mapsto -\log I(u_t)$ is convex. 
\end{claim}
\begin{proof}
    Let $v_0,v_1 \in \cC^+$. Define $v_t = (1-t)v_0 + tv_1$, which is also contained in $\cC^+$. 
    Set $u_t = (v_t)^*$ to be the Legendre transform. 
    The function $u(x,t) = u_t(x)$ is the partial Legendre transform of $u(x,t)$ and is jointly convex in $(x,t)$. 
    Since $-\log(F(-u))$ is convex by assumption, we have 
    \[
    I(u_t) = \int F(-u_t)\chi_{\{u_t <0\}}\, dx = \int_{\mathbb{R}^n} e^{-(-\log F(-u_t))}\, dx,
    \]
    which is convex by Prekopa's theorem (see e.g., \cite{CorderoKlartag-Prekopa}).
\end{proof}

\begin{claim}
    The function $G\circ J$ is convex on $\cC^+$ if $g' \geq 0$, and concave if $g' < 0$.
\end{claim}
\begin{proof}
    Let $v_t = (1-t)v_0 + tv_1$ for $v_i \in \cC^+$ and define $\Phi(t) = \frac{1}{H}\int_P G(v_t(y)) \, dy^h  = G(J(v_t))$, so 
    \begin{equation}\label{eqn:PhiDerivatives}    \dot{\Phi}(t) = \frac{1}{H}\int_P g(v_t(y))\dot{v}_t(y) \, dy^h\quad \text{and}\quad \ddot{\Phi}(t) = \frac{1}{H}\int_P g'(v_t(y))(\dot{v}_t(y))^2 \, dy^h,
    \end{equation}
    where we observe that each term in the second derivative is positive except $g'$, proving the claim.
\end{proof}

\begin{claim}\label{claim:convexityOfJ}
    If $g' \geq 0$, then $J$ is convex, and conversely, if $g' < 0$, then $J$ is concave.
\end{claim}
\begin{proof}
    Let $g' \geq 0$, which implies that $G$ is convex.
    For $v_0,v_1 \in \cC^+$, define $v_t = (1-t)v_0 + tv_1$.
    We will show that $J(t) = J(v_t)$ is convex for any such family, which is equivalent to 
    \[
    \frac{1}{H}\int_P G(v_t(y)) \, dy^h \leq G\left((1-t)G^{-1}\left(\frac{1}{H}\int_P G(v_0(y)) \, dy^h \right) + tG^{-1}\left(\frac{1}{H}\int_P G(v_1(y)) \, dy^h\right)\right)
    \]
    since $G^{-1}$ is strictly increasing. 
    We can use Jensen's inequality to show
    \[
    \frac{1}{H}\int_P G(v_t(y)) \, dy^h  \leq G\left(\frac{1}{H}\int_P v_t \, dy^h\right) = G\left((1-t)\frac{1}{H}\int_P v_0 \, dy^h + t\frac{1}{H}\int_P v_1 \, dy^h\right).
    \]
    Another application of Jensen's inequality gives
    \[
     \frac{1}{H}\int_P v(y) \, dy^h \leq G^{-1}\left(\frac{1}{H}\int_P G(v(y))\, dy^h\right),
     \]
      from which the result follows. The concave version is identical, and so we omit the proof. 
\end{proof}

An immediate corollary of Claims~\ref{claim:convexityOfI} and~\ref{claim:convexityOfJ} is:
\begin{cor}\label{cor:EisConvex}
    If $F$ is log-concave and $g' \geq 0$, then $\cE$ is convex.
\end{cor}

\section{Compactness}\label{sec3:Compactness}

In this section, we use Properties~$\mathbf{H}$ to show that the variational problem satisfies a suitable compactness principle.
The first issue we must address is that the energy functional $-\log I(v)$ is invariant under variations of the form $v \mapsto v-\langle x_0,y\rangle$ where $x_0 \in \Omega$, while $J(v)$ is not invariant. As explained in the introduction, this asymmetry leads immediately to a source of non-compactness for the variational problem. We show that, under the assumptions of \refHa, \refHb, or \refHc\ we can normalize a given $v \in \cC^+$ over all such variations; see Lemma~\ref{lem:interiorNormalizationPole}. 
We then show that the set of normalized functions with an energy upper bound $\cE(v) \leq A$ is compact. For the case of Property~\refHa, this is rather straightforward; see Proposition~\ref{prop:ConvexGELowerBound}. 
For pairs $(F,G)$ satisfying Property \refHb\ or \refHc\ the argument depends on (essentially sharp) integral estimates established in Proposition~\ref{prop: G-bound}, followed by a case-by-case analysis ruling out divergent sequences; see Proposition~\ref{prop:EEhatcomparisonConcaveb}. 

We recall the following key Lemma, due to John:
\begin{lem}[John]
    Let $\Omega$ be a bounded convex body. 
    There exists some ellipsoid $E$, called the John ellipsoid, centered at the origin and a point $x_0 \in \Omega$ such that 
    \[
    x_0 + E \subset \Omega \subset x_0 + nE.
    \]
    The set $\Omega$ is said to be in John's position if $x_0$ is the origin.
\end{lem}

\subsection{Normalizing the free boundary}
In this section we explain how the structural assumptions control the translational non-compactness of the variational problem. 
Recall from Section~\ref{subsec: translationsIntro} that the functional $I$ is invariant under translations $u(x) \mapsto u(x+x_0)$, which amounts to subtracting a supporting hyperplane from $v=u^*$. 
If $g(s)={\rm const.}$, and $P$ has $h$-barycenter at the origin, then $J$ is also invariant under these transformations. 
In this case, we have the freedom to normalize $v$ in several ways; for example, we may assume that $\Omega$ is in John's position or that $v(0)=\inf_Pv$.

On the other hand, when Properties~\refHb\ or~\refHc\ are in effect, the functional $J$ is no longer invariant, and the existence of a critical point over such variations depends on the structural assumptions.

We define the class of normalized functions 
\[
    \cC_P = \left\{v \in \cC^+: \int_P g(v(y))\vec{y} \, dy^h = 0\right\},
\]
which are critical points of $J$ with respect to linear variations.

\begin{lem}\label{lem:interiorNormalizationPole}
    Let $v \in \cC^+$. If $(F,G)$ satisfy Property \refHb${}_{\nu,\gamma}$\ or \refHc${}_{\nu,\gamma, \vec{\beta}}$\ for some $\nu \geq \mathtt{v}_0(h)$, then the map $x\mapsto G(J(v - \la x,\cdot \rg))$ is maximized at an interior point $x\in \Omega$, and $v - \la x,\cdot\rg \in \cC_P$.
\end{lem}
\begin{proof}
    The Legendre dual of $\cC^+$ consists of functions $u$ satisfying $\phi_P - C \leq u < \phi_P$, so $\Omega$ lies inside some dilate of $P^\circ$, and in particular, $\Omega \subset B_R(0)$ is bounded. We claim that if $x\in \del \Omega$, then
    \[
    \int_{P} G(v-\langle x,y\rangle)\, dy^h =-\infty.
    \]
    Indeed, if $x\in\del \Omega$, then $v_{x}(y)=v-\langle x,y\rangle$ vanishes at some point $p\in \overline{P}$. 
    Near $p$, we have
    \[
    v(y) \leq R|p-y|\quad \text{ and }\quad h(y)  \gtrsim |y-p|^{\mathtt{v}_o},
    \]
    and the result follows from Lemma~\ref{lem: poleNorm}.

\end{proof}

\begin{remark}
    The general strategy and estimates established throughout the paper hold for pairs $(F,G)$ with $G$ an arbitrary convex function. However, we have not been able to identify a condition on a (non-linear) convex function $G$ which guarantees that the function
    \begin{equation}\label{eq: rkInteriorMin}
    \{v^*<0\} \ni x \mapsto \int_{P} G(v-\langle x,y\rangle)\, dy^h
    \end{equation}
    has an interior minimum.  In general, it is not hard to construct examples of non-linear convex functions $G$, convex sets $P$, and functions $v$ for which the minimum in~\eqref{eq: rkInteriorMin} is achieved on $\del\{v^*<0\}$, even when $P$ has $h$-barycenter at the origin. 
\end{remark}

\subsection{Compactness for Property \texorpdfstring{\refHa}{(H1)} }

\begin{prop}\label{prop:ConvexGELowerBound}
    Suppose that $P$ has $h$-barycenter at the origin. For a pair $(F,G)$ satisfying Property~\refHa , there is a constant $C>1$ depending only on universal data so that
    \[
    \mathcal{E}_{\Lambda}(v) \geq -\log( F(v(0))|\Omega|) + \frac{C^{-1}\Lambda}{H} {\rm diam}(\Omega).
    \]
    In particular, there is a $\Lambda_0$, depending only on universal data, such that, for all $\Lambda > \Lambda_0$ we have:
    \begin{itemize}
        \item [(ii)] For all $v$ such that $\cE_{\Lambda}(v) \leq A$, there exist uniform constants $\delta_*, D^*$ depending on $A$ and universal data such that 
    \[
    \delta_* \leq v(0) \lesssim|\Omega|^{\frac{1}{n}} \lesssim {\rm diam}(\Omega) \leq D^*.
    \] 
        \item[(ii)] There is an energy lower bound $\inf_{v \in \cC^+}\cE_{\Lambda}(v) >-\infty$.
    \end{itemize}
    
\end{prop}
\begin{proof}
    For any $v \in \cC_P$, we have 
    \[
    \cE(v) = -\log(I(v)) + \frac{\Lambda}{H}\int_P v(y)\, dy^h. 
    \]
    Observing that the right-hand side is translation invariant, since $P$ has $h$-barycenter at the origin, we may assume that $\Omega$ is in John's position. 
    Let $E$ be the John ellipsoid of $\Omega$, so $E\subset \Omega \subset nE$. Let $L_1 \leq L_2 \leq \cdots \leq L_n$ be the eccentricities of $E$, which we may assume to be parallel to the coordinate axes. Let $ \phi_{E} = \sqrt{\sum_{i=1}^{n}L_i^2y_i^2}$ be the convex support function of $E$. Then $v \geq \phi_{\Omega} \geq \phi_{E}$, and we have
    \[
    \int_P v(y)\, dy^h \geq \int_{\frac{1}{2}P} \phi_{E}\, dy^h \geq \inf_{\frac{1}{2}P}h \int_{\frac{1}{2}P}\phi_{E}\,dy.
    \]
    The latter integral can be computed in polar coordinates
    \[
    \int_{\frac{1}{2}P}\phi_{E}dy = \int_{y\in \bS^{n-1}}\phi_{E}\left(y\right) \int_{0}^{R(y)}r^{n}\,dr d\omega(y) \gtrsim \int_{y\in \bS^{n-1}}\phi_{E}(y)\, d\omega(y).
    \]
    On the set $ \{y \in \bS^{n-1}: |y_n| > \frac{1}{2}\}$, we have
    \[
    \phi_{E}\left(y\right) \geq \frac{1}{2}L_n,
    \]
    and so
  \[
  \begin{aligned}
  \int_{\bS^{n-1}}\phi_{E}(y)\,d\omega(y) \gtrsim L_n .
  \end{aligned}
  \]
  Combining with Lemma~\ref{lem: IBound}, we obtain
  \begin{equation}\label{eq: H1Energylowbnd}
  \mathcal{E}(v) \geq -\log(F(v(0))|\Omega|) + \frac{C^{-1}\Lambda}{H}{\rm diam}(\Omega).
  \end{equation}
  We now show that compactness follows from this estimate.
  Since $v(0)= |\inf_{\Omega} u|$, and $\nabla u(\Omega) \subset P$, the ABP inequality yields
  \[
  v(0) \lesssim |\Omega|^{\frac{1}{n}} \lesssim {\rm diam}(\Omega).
  \]
  Thus, if we set $x \sim {\rm diam}(\Omega)$, then we have
  \[
\mathcal{E}(v) \geq   -\log(F(x)x^n) + \Lambda C^{-1}x-C
  \]
  for some constant $C$. 
  By Property~\refHa, $-\log(F(x)) \geq  -Mx$ for some $M$. 
  Choosing $\Lambda$ large, we see that $\mathcal{E}(v)$ is uniformly bounded from below. 
  Furthermore, if $\mathcal{E}(v) \leq A$, then we see that ${\rm diam}(\Omega)$ is bounded from above by some uniform constant $D^*$, and hence so is $v(0)$. Finally, since $|\Omega|$ is bounded from above and $F(0)=0$, it follows from~\eqref{eq: H1Energylowbnd} again, that $v(0) \geq \delta_* >0$ for some uniform $\delta_*$. 
\end{proof}

\subsection{Compactness for Property \texorpdfstring{\refHb}{(H2)} and \texorpdfstring{\refHc}{(H3)}}

We now turn our attention to the cases when $(F,G)$ satisfy Property~\refHb \, and~\refHc.
The main difficulty is to prove sharp estimates for the functional $J$.

Let $\Omega$ be a convex body in John's position, so there is an ellipsoid $E$ centered at the origin such that $E \subset \Omega \subset nE$. Up to a linear map, we may assume that
\[
E = \left\{ y \in \mathbb{R}^n : \sum \frac{y_i^2}{L_i^2} \leq 1\right\}
\]
where $0<L_1\leq\cdots\leq L_n$. 

We will need the following elementary, but useful result which says that when $\Omega$ is in John's position, $\inf_P v$ and $v(0)$ are comparable.

\begin{lem}\label{lem:JohnsOmegaBalanced}
    Suppose that $u$ is a convex function where $\Omega = \{ u < 0\}$ is in John's position.  Then, $\frac{ v(0)}{n+2} \leq \inf_P v$. 
\end{lem}
\begin{proof}
    Let $x_0$ achieve the infimum of $u$. 
    Consider the line $\overline{0x_0}$ through $0$ and $x_0$, which intersects $\p E$ at $y_1$ and $y_2$, with $|y_1|=|y_2|$.  We assume that $|x_0-y_2|\leq |x_0-y_1|$. 
    On $\overline{0x_0}$, we know that $u$ lies below the piecewise linear function $\ell$ defined by $\ell(x_0) = u(x_0)$ and $\ell(y_i) = 0$.
    Therefore, we have 
    \[
    \frac{|u(0)|}{|u(x_0)|} \geq 1 - \frac{|x_0|}{d(x_0,y_1)}.
    \]
    Rearranging the above shows that the desired inequality is equivalent to $\frac{1-c}{c}|y_1| \geq |x_0|$.  On the other hand, $|x_0| \leq n|y_1|$ since $\Omega \subset nE$.    Therefore choosing $c = \frac{1}{n+2}$ is sufficient. 
\end{proof}
Consider some $u = v^*$ for $v \in \cC^+$ and $\Omega = \{u < 0\}$. Let $E$ be the John ellipsoid of $\Omega$ and let $L_1 \leq \cdots \leq L_n$ be the lengths of the principal axes of $E$. Let $x_0$ be the point achieving the infimum of $u$. By convexity we have
\[
\frac{|u(x_0)|}{d(x_0,\p \Omega)} \leq C(P), \quad  d(x_0, \del \Omega) \leq nL_1,
\] 
and so, in particular 
\begin{equation}\label{eq: infuControlsSmallEcc}
\inf_{P}v \lesssim L_1 \leq L_2 \leq \cdots \leq L_n.
\end{equation}
The following proposition is one of the key technical results of the paper.

\begin{prop}\label{prop: G-bound}
Let $u=v^{*}$ for $v\in \cC^+$ and suppose that $\Omega = \{ u<0\}$ is in John's position. Let $\rho^{-}, \rho^+$ denote the inner and outer radii of $P$, respectively.

There is a uniform constant $C>1$, depending only on universal data such that if we set $\epsilon_2 =\frac{\rho^-}{2n^{3/2}\rho^+}$, $x= \inf_P v$, and $D=L_n$, then the following estimates hold:
\begin{itemize}
\item[(i)] If \refHb\ holds with $\nu >0$ \and $\epsilon_2x \leq 1$, then
    \[
    G(J(v)) \geq CG(\epsilon_2 x)\frac{x^{n}}{|\Omega|} - \frac{C}{|\Omega|}x^{-\nu} + C^{-1}\max\{0, G(n\rho^+\epsilon_2D)\}.
    \]
    \item[(ii)] If \refHb\ holds with $\nu =0$ \and $\epsilon_2x \leq 1$, then
    \[
    G(J(v)) \geq CG(\epsilon_2 x)\frac{x^{n}}{|\Omega|} - \frac{C}{|\Omega|}(|\log x|+1) + C^{-1}\max\{0, G(n\rho^+\epsilon_2D)\}.
    \]
    \item[(iii)] If \refHb\ holds and $\epsilon_2x \geq 1$, then
    \[
    G(J(v)) \geq C^{-1}G(\epsilon_2 x)\frac{x^{n}}{|\Omega|} + C^{-1}\max\{0, G(n\rho^+\epsilon_2D)\}.
    \]
     \item[(iv)] If \refHc\ holds with $\nu>0$, and $x \leq 1$, then
    \[
    G(J(v)) \geq CG(x)\frac{x^{n}}{|\Omega|} - \frac{C}{|\Omega|}x^{-\nu} + \frac{C}{D } \int_{1}^{\infty}G(s)s^{n-1}\,ds.
    \]
    \item[(v)] If \refHc\ holds with $\nu=0$, and $x \leq 1$, then
    \[
    G(J(v)) \geq CG( x)\frac{x^{n}}{|\Omega|} - \frac{C}{|\Omega|}(|\log x|+1)  + \frac{C}{D} \int_{1}^{\infty}G(s)s^{n-1}\,ds.
    \]
    \item[(vi)] If \refHc\ holds and $x >1$, then
    \[
    G(J(v)) \geq CG(x)\frac{x^{n}}{|\Omega|}+ \frac{C}{D x^{n-1}} \int_{n^{-3/2}x}^{\infty}G(s)s^{n-1}\,ds.
    \]
    \end{itemize}
    \end{prop}

\begin{proof}
Let $E$ be the John ellipsoid of $\Omega$.
By a rotation, we may assume that the principal axes of $E$ are parallel to the coordinate axes.
We let the principal axes be given by $0<L_1 \leq \cdots \leq L_n$.
Then, for any $\epsilon \in (0,1]$ we have
\begin{equation}\label{eq: vLowBoundMax}
v \geq \max\{\phi_{\Omega}, \epsilon \inf_Pv\} \quad \text{ and } \quad \phi_{E} \leq \phi_{\Omega} \leq n\phi_{E},
\end{equation}
where $\phi_{E} = \sqrt{ \sum_{i=1}^{n} L_i^2y_i^2}$. 
We consider the cone with rectangular cross-section given by
\[
\Gamma_i = \{ y: L_i|y_i| \geq \max_{j\ne i} L_j|y_j|\}.
\]
In $\Gamma_i$ we have the estimate
\begin{equation}\label{eq: PhiEUpLowBound}
|y_i|L_i \leq \phi_{E} \leq \sqrt{n}|y_i|L_i, \qquad y \in \Gamma_i.
\end{equation}
We now estimate, for any $\epsilon \in (0,1]$,
\begin{equation}\label{eq: innerEstimate}
\begin{aligned}
\int_{P}G(v(y))\,dy^h &\geq \int_P G(\max\{\phi_{\Omega}, \epsilon \inf_Pv\})\,dy^h\\
&\geq C\epsilon^nG(\epsilon \inf_P v)(\inf_P v)^n|\Omega^{\circ}| + \int_{P\cap\{ \phi_{\Omega} \geq \epsilon \inf_P v\}}G(\phi_{\Omega})\,dy^h.
\end{aligned}
\end{equation}
To estimate the second integral, we write
\[
\begin{aligned}
\int_{P\cap\{ \phi_{\Omega} \geq \epsilon \inf_P v\}}G(\phi_{\Omega})\,dy^h &= \sum_{i=1}^{n}\int_{\Gamma_i \cap P\cap\{ \phi_{\Omega} \geq \epsilon \inf_P v\}}G(\phi_{\Omega})\,dy^h\\
& \geq \sum_{i=1}^{n}\int_{\Gamma_i \cap P\cap\{ \phi_{\Omega} \geq \epsilon \inf_P v\}}G(L_i|y_i|)\,dy^h.\\
\end{aligned}
\]
We now need to simplify the region. 
Since we wish to bound the integral from below, we need to consider cases \refHb\ and \refHc\ separately.
The appropriate choice of $\epsilon$ will depend on whether structural conditions \refHb\ or \refHc\ are in effect. 
We denote by $\epsilon_2$ the constant in case \refHb\ holds, and $\epsilon_3$ the constant in case \refHc\ holds. 
Recall we are considering the region $\{\phi_{\Omega} \geq \epsilon \inf_{P}v\}$. 

Thanks to the bounds~\eqref{eq: vLowBoundMax} and~\eqref{eq: PhiEUpLowBound}, we have the containment 
\[
\{\epsilon \inf_P v \leq |y_i|L_i \} \cap \Gamma_i \subset \{\epsilon \inf_{P}v \leq \phi_{\Omega}\}\cap \Gamma_i \subset \{\epsilon \inf_P v \leq n^{3/2}|y_i|L_i\} \cap \Gamma_i.
\]
For convenience, let us set 
\[
L_0 := \frac{1}{n^{3/2}}\inf_{P}v.
\]
Let $Q_{r}$ denote the cube centered at the origin with side length $2r$, so
\[
Q_{\frac{\rho^-}{\sqrt{n}}} \subset P \subset Q_{\rho^{+}}.
\]
If Property~\refHb\ holds, then we have
\[
\begin{aligned}
\int_{\Gamma_i \cap P\cap\{ \phi_{\Omega} \geq \epsilon_2 \inf_P v\}}G(L_i|y_i|)\,dy^h &\geq C \underbrace{\int_{\Gamma_i \cap Q_{\rho^+}\cap\{ \epsilon_2 L_0 \leq |y_i|L_i \leq 1\} }G(L_i|y_i|)\,dy}_{(A2)_i}\phantomsection\label{(A2)_i} \\
&\quad+ \underbrace{\int_{\Gamma_i \cap P\cap\left\{ |y_i|L_i \geq \max\{1, \epsilon_2 \inf_Pv\}\right\}}G(L_i|y_i|)\,dy^h}_{(B2)_i}. 
\end{aligned}
\]
Choose $\epsilon_2 = \frac{\rho^-}{2n^{3/2}\rho^+}$.
We claim that if $L_i\frac{\rho^{-}}{2\sqrt{n}}>1$, then
\begin{equation}\label{eq: usefulContainment}
\Gamma_i \cap Q_{\frac{\rho^{-}}{\sqrt{n}}}\cap\left \{ |y_i| > \frac{\rho^{-}}{2\sqrt{n}}\right\} \subset \Gamma_i \cap P\cap\left\{ |y_i|L_i \geq \max\{1, \epsilon_2 \inf_Pv\}\right\}.
\end{equation}
To see this, first note that by the intermediate value theorem we have
\[
\frac{\inf_P v}{nL_i}
= \frac{-u(0)}{nL_i} \leq \rho^+,
\]
and so $L_i \geq \frac{\inf_P v}{n\rho^+}$. Thus, if $|y_i| \geq \frac{\rho^{-}}{2\sqrt{n}}$ and $L_i\frac{\rho^{-}}{2\sqrt{n}}>1$, we have $ |y_i|L_i \geq \max\{1,\epsilon_2 \inf_P v\}$. 
However, since $P\supset Q_{\frac{\rho^{-}}{\sqrt{n}}}$, we conclude that~\eqref{eq: usefulContainment} holds.

If Property~\refHc\ holds, we take $\epsilon_3=1$ and we have
\[
\begin{aligned}
\int_{\Gamma_i \cap P\cap\{ \phi_{\Omega} \geq \inf_P v\}}G(L_i|y_i|)\,dy^h &\geq C\underbrace{\int_{\Gamma_i \cap Q_{\rho^+}\cap\{ L_0 \leq |y_i|L_i \leq 1\} }G(L_i|y_i|)\,dy}_{(A3)_i} \phantomsection \label{(B3)_i}\\
&\quad + C\underbrace{\int_{\Gamma_i \cap Q_{\rho^+}\cap\left\{ |y_i|L_i \geq \max\{1, L_0\}\right\}}G(L_i|y_i|)\,dy}_{(B3)_i}
\end{aligned}
\]

We will only estimate the contribution of the integrals $\mr{(}$\hyperref[(A2)_i]{$A2/3$}$\mr{)}_i$, $\mr{(}$\hyperref[(A2)_i]{$B2/3$}$\mr{)}_i$ on the region where $y_i>0$, as the region $y_i<0$ is treated identically. 

Consider the set $\{L_iy_i >L_j|y_j|\}\subset \mathbb{R}^2$, the boundary of which is the line $L_iy_i = L_j|y_j|$.
This line will intersect the boundary of $Q_{r}$ at a point with $y_i \leq r$ if and only if $L_j \leq L_i$. 
The first such point will be when $\frac{L_1r}{L_i}= y_i$.
In particular, we see that the set $\Gamma_i \cap Q_r\cap \{0<y_i \leq \frac{L_1r}{L_i}\} $ is a segment of an $n$-dimensional cone in the $y_i$ direction.
On the other hand, for $j \geq 2$, the set $\Gamma_i \cap Q_r\cap \{\frac{L_{j-1}r}{L_i}<y_i \leq \frac{L_jr}{L_i}\} $ is the product of a segment of an $(n-j+1)$-dimensional cone over the $y_i$ direction, and the $(j-1)$-dimensional cube $\{|y_1|<r\}\times \cdots \times \{|y_{j-1}|<r\}$.
We use this to estimate the integrals $\mr{(}$\hyperref[(A2)_i]{$A2/3$}$\mr{)}_i$ and $\mr{(}$\hyperref[(A2)_i]{$B2/3$}$\mr{)}_i$.
We begin by estimating the integrals $\mr{(}$\hyperref[(A2)_i]{$A2/3$}$\mr{)}_i$, which can be treated using the same argument.
First note that if $\epsilon L_0 \geq 1$, then $|\Gamma_i \cap Q_{\rho^+}\cap \{\epsilon L_0 \leq y_iL_i \leq 1\}|=0$, and so $\mr{(}$\hyperref[(A2)_i]{$A2/3$}$\mr{)}_i =0$.
Thus, we will assume that $\epsilon L_0 < 1$.
Let
\[
k_- = \max\left\{ i : \rho^+L_i \leq 1\right\}\qquad \text{and} \qquad G_{-} = \min\{G, 0\}.
\]
We can then estimate 
\[
\begin{aligned}
\mr{(} \hyperref[(A2)_i]{A2/3} \mr{)}_i &= \int_{\Gamma_i \cap Q_{\rho^+}\cap \{\epsilon L_0 \leq y_iL_i \leq 1\}} G(L_iy_i)\, dy\\
&\gtrsim \prod_{\ell=1}^{n} \left(\frac{L_i}{L_\ell}\right) \int_{\frac{\epsilon L_{0}}{L_i}}^{ \frac{\min\{1, L_1\rho^+\}}{L_i}} G_{-}(L_iy_i) y_i^{n-1}\,dy_i\\
&\quad  + \sum_{j=2}^{\min\{k_-,i-1\}} (2\rho^+)^{j-1}\left(\prod_{\ell \geq j} \frac{L_i}{L_\ell} \right)\int_{\frac{L_{j-1}\rho^+}{L_i}}^{\frac{L_j\rho^+}{L_i}} G_-(L_iy_i) y_i^{n-j}\,dy_i\\
&\quad  + (2\rho^+)^{i-1}\left(\prod_{\ell \geq i} \frac{L_i}{L_\ell}\right) \int_{\frac{\min\{ 1, L_{i-1}\rho^+\}}{L_i}}^{\frac{1}{L_i}}G_-(L_iy_i)y_i^{n-i}\,dy.\\
\end{aligned}
\]
Changing variables leads to
\[
\begin{aligned}
\mr{(} \hyperref[(A2)_i]{A2/3} \mr{)}_i &\gtrsim \prod_{\ell=1}^{n} \left(\frac{1}{L_\ell}\right)\int_{\epsilon L_{0}}^{\min\{1,L_1\rho^+\}} G_-(s) s^{n-1}\,ds \\
&\quad + \sum_{j=2}^{i-1}(2\rho^+)^{j-1} \prod_{\ell \geq j} \frac{1}{L_\ell} \int_{L_{j-1}\rho^+} ^{L_j\rho^+} G_-(s) s^{n-j}\,ds\\
&\quad + (2\rho^+)^{i-1}\left(\prod_{\ell \geq i} \frac{1}{L_\ell}\right) \int_{\min\{1, L_{i-1}\rho^{+}\}}^{1}G_-(s)s^{n-i}\,ds.
\end{aligned}
\]
By Property~\refHb\, or Property~\refHc\, part (2), we have $G(s) \gtrsim -s^{-(n+\nu)}$ for $s\leq 1$ and so, for $0<a <b <1$, we get
\[
\int_{a}^{b} G_-(s)s^{n-k-1}\,ds \gtrsim -\int_{a}^{b} s^{-(\nu+k+1)}\,ds= \begin{cases}
\log(\frac{a}{b}) & \text{ if } \nu=k=0,\\
\frac{1}{\nu+k}\left(b^{-(\nu+k)}- a^{-(\nu+k)}\right) & \text{ else.} 
\end{cases}
\]
Using this, we compute, for $2 \leq j\leq k_{-}$
\[
\left(\prod_{\ell \geq j}\frac{1}{L_\ell}\right) \int_{L_{j-1}\rho^+}^{L_j \rho^+}G_-(s)s^{n-j}\,ds \gtrsim -\left(\prod_{\ell \geq j}\frac{1}{L_\ell}\right)\left(L_{j-1}\rho^+\right)^{1-(\nu+j)}.
\]
On the other hand, since $L_{j-1}^{-(j-1)} \prod_{\ell \geq j} \frac{1}{L_k} \leq   \prod_{\ell \geq 1} \frac{1}{L_\ell}$, this implies
\[
\left(\prod_{k \geq j}\frac{1}{L_k}\right) \int_{L_{j-1}\rho^+}^{L_j \rho^+}G(s)s^{n-j}\,ds \gtrsim - \frac{1}{|\Omega|}L_{j-1}^{-\nu}.
\]
Plugging in this estimate yields
\[
\mr{(} \hyperref[(A2)_i]{A2/3} \mr{)}_i  \gtrsim \begin{cases} -\frac{1}{|\Omega|} (\inf_Pv)^{-\nu} -\frac{1}{|\Omega|}\sum_{j=1}^{\min\{i,k_-\}} L_j^{-\nu} & \text{ if } \nu >0,\\
 \frac{1}{|\Omega|} \log (\inf_Pv) -\frac{1}{|\Omega|} & \text{ if } \nu =0.
 \end{cases}
 \]
 Using that $L_j \gtrsim  \inf_Pv$, and summing over $i$ yields 
 \begin{equation}\label{eq: A23IntEst}
 \sum_{i=1}^{n} \mr{(} \hyperref[(A2)_i]{A2/3} \mr{)}_i  \geq \begin{cases} -\frac{C}{|\Omega|} \inf_Pv^{-\nu}  & \text{ if } \nu >0 \text{ and } \epsilon_{2/3} L_0 <1,\\
 -\frac{C}{|\Omega|}\left( |\log \inf_Pv| +1 \right) & \text{ if } \nu =0 \text{ and } \epsilon_{2/3} L_0 <1,\\
 0 & \text{ else}.
 \end{cases}
 \end{equation}

Now we consider the integrals $\mr{(}$\hyperref[(A2)_i]{$B2$}$\mr{)}_i$ and $\mr{(}$\hyperref[(B3)_i]{$B3$}$\mr{)}_i$.
The estimate we prove will depend on whether we are in case \refHb\ or case \refHc.
We consider \refHb\ first. 
In this case, assuming that $L_i\frac{\rho^{-}}{2\sqrt{n}}>1$, and using~\eqref{eq: usefulContainment}, and $\inf_{Q_{\frac{\rho^-}{\sqrt{n}}}}h(y)  >\delta >0$, we can estimate
\begin{equation}\label{eq: B2EstimateFirst}
\begin{aligned}
\mr{(} \hyperref[(A2)_i]{B2}\mr{)}_i &\gtrsim \int_{\Gamma_i \cap Q_{\frac{\rho^{-}}{\sqrt{n}}}\cap \{y_i >\frac{\rho^-}{2\sqrt{n}}\}} G(L_iy_i)\, dy\\
&\gtrsim \left(\frac{2\rho^-}{\sqrt{n}}\right)^i\left(\prod_{j \geq i} \frac{L_i}{L_j}\right) \int_{\frac{\rho^{-}}{2\sqrt{n}}}^{\frac{\rho^{-}}{\sqrt{n}}}G(L_iy_i)y_i^{n-i-1}\,dy_i\\
& = \left(\frac{2\rho^-}{\sqrt{n}}\right)^{i-1}\left(\prod_{j \geq i} \frac{1}{L_j}\right)\int_{L_i\frac{\rho^{-}}{2\sqrt{n}}}^{L_i\frac{\rho^{-}}{\sqrt{n}}}G(s)s^{n-i}\,ds\\
& \gtrsim \left(\frac{2\rho^-}{\sqrt{n}}\right)^{i-1}\left(\prod_{j \geq i} \frac{1}{L_j}\right)G\left(L_i\frac{\rho^{-}}{2\sqrt{n}}\right)\left(L_i\frac{\rho^{-}}{\sqrt{n}}\right)^{n-i+1}.
\end{aligned}
\end{equation}
On the other hand, if $L_i\frac{\rho^{-}}{2\sqrt{n}} \leq 1$, then we can bound $\mr{(}$\hyperref[(A2)_i]{$B2$}$\mr{)}_i \geq 0$. 
Summing over $i$, and using that $G\left(L_i\frac{\rho^{-}}{2\sqrt{n}}\right) \geq 0$, we may keep only the $i=n$ term in the sum, and so
\begin{equation}\label{eq: B2Estimate}
\sum_i \mr{(} \hyperref[(A2)_i]{B2}\mr{)}_i \geq C^{-1}\max\left\{0, G\left(\frac{\rho^{-}L_n}{2\sqrt{n}}\right)\right\} .
\end{equation}
Let us now summarize the estimate in case \refHb\ holds.
We sum the estimates from~\eqref{eq: innerEstimate}, \eqref{eq: A23IntEst}, and~\eqref{eq: B2Estimate}.
Since $\Omega$ has its John ellipsoid at the origin, we have $|\Omega|^{-1}\sim |\Omega^{\circ}|$ and so we obtain:
\begin{itemize}
    \item If \refHb\ holds with $\nu >0$ \and $\epsilon_2\inf_Pv \leq 1$, then
    \[
    G(J(v)) \geq CG(\epsilon_2 \inf_Pv)\frac{(\inf_P v)^{n}}{|\Omega|} - \frac{C}{|\Omega|}(\inf_Pv)^{-\nu} + C^{-1}\max\{0, G(n\rho^+\epsilon_2L_n)\}.
    \]
    \item If \refHb\ holds with $\nu =0$ \and $\epsilon_2\inf_Pv \leq 1$, then
    \[
    G(J(v)) \geq CG(\epsilon_2 \inf_Pv)\frac{(\inf_P v)^{n}}{|\Omega|} - \frac{C}{|\Omega|}(|\log \inf_Pv|+1) + C^{-1}\max\{0, G(n\rho^+\epsilon_2L_n)\}.
    \]
    \item If \refHb\ holds and $\epsilon_2\inf_Pv \geq 1$, then
    \[
    G(J(v)) \geq \frac{1}{C}G(\epsilon_2 \inf_Pv)\frac{(\inf_P v)^{n}}{|\Omega|} + C^{-1}\max\{0, G(n\rho^+\epsilon_2L_n)\}.
    \]
\end{itemize}

Now assume we are in case \refHc; in particular
\[
\int_{1}^{\infty}G(s)s^{n-1}\,ds <+\infty.
\]
Note that if $\rho^+L_i \leq \max\{L_0,1\}$, then $Q_{\rho^+}\cap \{y_iL_i >\max\{L_0,1\}\}= \emptyset$. In other words, we may assume that $i \geq k_-$ and $\frac{1}{L_i} \leq \frac{\rho^+}{\max\{L_0,1\}}$. 
In this case, we can bound $\mr{(}$\hyperref[(B3)_i]{$B3$}$\mr{)}_i$ as follows:
\[
\begin{aligned}
\mr{(}\hyperref[(B3)_i]{B3}\mr{)}_i &\gtrsim  \int_{\Gamma_i \cap Q_{\rho^+}\cap \{y_iL_i >\max\{1,L_0\}\}} G(L_iy_i)\, dy\\
& \gtrsim (\rho^+)^{i-1}\left(\prod_{j \geq i} \frac{L_i}{L_j}\right) \int_{\frac{\max\{1,L_0\}}{L_i}}^{\rho^+}G(L_iy_i)y_i^{n-i}\,dy_i\\
&=(\rho^+)^{i-1}\left(\prod_{j \geq i} \frac{1}{L_j}\right) \int_{\max\{1,L_0\}}^{L_i\rho^+}G(s)s^{n-i}\,ds\\
& \geq (\rho^+)^{i-1}\left(\prod_{j \geq i} \frac{1}{L_j}\right) \left(\frac{1}{\max\{1,L_0\}}\right)^{i-1}\int_{\max\{1,L_0\}}^{L_i\rho^+}G(s)s^{n-1}\,ds\\
& \gtrsim -\frac{1}{L_n \max\{L_0,1\}^{n-1}} \int_{\max\{1,L_0\}}^{\infty}G(s)s^{n-1}\,ds .
\end{aligned}
\]
We can now summarize the estimate in the case that \refHc\ holds:
\begin{itemize}
    \item If \refHc\ holds with $\nu>0$, and $\inf_Pv \leq 1$, then
    \[
    G(J(v)) \geq CG(\inf_Pv)\frac{(\inf_P v)^{n}}{|\Omega|} - \frac{C}{|\Omega|}(\inf_Pv)^{-\nu} + \frac{C}{L_n} \int_{\max\{1,n^{-3/2}\inf_{P}v\}}G(s)s^{n-1}\,ds.
    \]
    \item If \refHc\ holds with $\nu=0$, and $\inf_Pv \leq 1$, then
    \[
    G(J(v)) \geq CG( \inf_Pv)\frac{(\inf_P v)^{n}}{|\Omega|} - \frac{C}{|\Omega|}(|\log \inf_Pv|+1)  + \frac{C}{L_n} \int_{\max\{1,n^{-3/2}\inf_{P}v\}}G(s)s^{n-1}\,ds.
    \]
    \item If \refHc\ holds and $\inf_Pv >1$, then
    \[
    G(J(v)) \geq CG( \inf_Pv)\frac{(\inf_P v)^{n}}{|\Omega|}+ \frac{C}{L_n (\inf_Pv)^{n-1}} \int_{\max\{1,n^{-3/2}\inf_{P}v\}}G(s)s^{n-1}\,ds.
    \]
\end{itemize}
\end{proof}

Using the estimates of Proposition~\ref{prop: G-bound}, we can now prove that, under Property \refHb\ or \refHc, the variational problem satisfies a compactness principle.

\begin{prop}\label{prop:EEhatcomparisonConcaveb}
    Suppose that the pair $(F,G)$ satisfies structural Property \refHb\ or \refHc. There exists a universal constant $\Lambda_0 \geq 0$ with the following properties:
    \begin{itemize}
        \item[(i)] For any $\Lambda > \Lambda_0$ and any $v \in \cC_P$, we have
        \[
        \cE_{\Lambda}(v) \geq -C_{\Lambda}>-\infty.
        \]
        \item[(ii)] Fix $A \in \mathbb{R}$. There exist uniform constants $\delta_*,D^*$, depending only on the structure, and on $A, \Lambda$, such that, if $v \in \cC_P$ satisfies $\cE_{\Lambda}(v) \leq A$, then
        \[
        \delta_* \leq v(0) \lesssim|\Omega|^{\frac{1}{n}} \lesssim {\rm diam}(\Omega) \leq D^*.
        \]
        \item[(iii)] If $(F,G)$ satisfy Property \refsHb\ or \refsHc, then we can take $\Lambda_0=0$. 
    \end{itemize}
\end{prop}
\begin{proof}
Let $\widetilde{v}(y) = v(y) - \la x, y\rg$ such that $\widetilde{\Omega} = \Omega - x$ has its John ellipsoid $E$ centered at the origin.
From Lemma~\ref{lem:JohnsOmegaBalanced}, we know that $\inf_P \widetilde{v} \sim \widetilde{v}(0)= v(0)$ and $|\Omega|^{-1} \sim |\widetilde{\Omega}^\circ|$.
By definition of $\cC_P$ we have
\[
A \geq \mathcal{E}(v) \geq \mathcal{E}(\widetilde{v})  \geq -\log \left(F(v(0))|\Omega|\right) +\Lambda J(\widetilde{v}) -C.
\]
We now prove bounds for $v(0)$ and $|\Omega|$.
\\

\noindent {\bf Step 1: $v(0)$ is bounded above.}
\\

\noindent First observe that $v(0) \lesssim |\Omega|^{\frac{1}{n}}$ by the ABP estimate.
We now analyze the behavior of $\mathcal{E}(v)$ depending on whether structural conditions \refHb\ or \refHc\ hold.
Assume first that \refHb\ holds. Since $\inf_P \widetilde{v} \sim v(0)$, for $v(0)$ sufficiently large, depending only on universal data, we may apply Proposition~\ref{prop: G-bound} $(iii)$. To avoid confusion, we set $y=v(0)$, adopting the notation $x =\inf_P\widetilde{v}$ from Proposition~\ref{prop: G-bound}, and using that $x \gtrsim y$ by Lemma~\ref{lem:JohnsOmegaBalanced}.
As in Proposition~\ref{prop: G-bound}, we use $D$ to denote the length of the major axis of $E$. 
Using that $G$ is increasing, and $G(s) \geq 0$ for $s \geq 1$, we have
\[
\begin{aligned}
J(v) &\geq G^{-1}\left(C\frac{G(\epsilon y)y^n}{|\Omega|} + C^{-1}G(n\rho^{-}\epsilon_2 D)\right)\\
& \geq  G^{-1}\left(C^{-1}G(n\rho^{-}\epsilon_2 D)\right).
\end{aligned}
\]
Using $|\Omega| \lesssim D^{n}$, we claim that
\[
\lim_{D \rightarrow \infty}-n\log D + \Lambda G^{-1}(\epsilon G(\epsilon D)) =\infty
\]
provided $\Lambda > \Lambda_{\epsilon}$ is sufficiently large. Indeed, since $F(y) \geq cy $ for some $c>0$, the claim follows from the structural condition \refHb\ part $(4)$. Thus, if $v(0)$ is sufficiently large we have
\[
\mathcal{E}(v) \geq -\log F(y) + \Lambda G^{-1}(\epsilon G(\epsilon y)),
\]
using that $D \gtrsim y=v(0)$. Again, it follows from structural condition \refHb\ part $(4)$ that $v(0)$ is bounded provided $\Lambda$ is chosen sufficiently large. If \refsHb\ holds, then $v(0)$ is bounded for any $\Lambda>0$. 
Thus, we conclude that for $\Lambda > \Lambda_0$ sufficiently large, $v(0)$ is bounded from above depending only on $A, \Lambda$, and we may take $\Lambda_0=0$ in case \refsHb\ holds.
\\

\noindent Now assume that \refHc\ holds. 
Then, by Proposition~\ref{prop: G-bound} $(vi)$ we have
\[
J(v) \geq \begin{cases} G^{-1}\left(-C\frac{1}{|\Omega|}  - \frac{1}{y^{n-1}D} o(1)\right) &\text{ if } \beta_1=0,\\
G^{-1}\left(-\frac{Cy^{-\beta_1}}{|\Omega|}  - \frac{C}{y^{n+\beta_1-1}D}\right) &\text{ if } \beta_1>0.
\end{cases}
\]
Since $y \lesssim|\Omega|^{\frac{1}{n}}$ by the ABP estimate, we have
\[
J(v) \gtrsim y^{\frac{\beta_1}{n+\beta_2}}|\Omega|^{\frac{1}{n+\beta_2}} 
\]
provided $v(0)$ is sufficiently large.
Now we observe that 
\[
-\log|\Omega| + \frac{\Lambda}{2} J(v) \geq -\log|\Omega| +\frac{\Lambda}{2}y^{\frac{\beta_1}{n+\beta_2}}|\Omega|^{\frac{1}{n+\beta_2}}  \rightarrow +\infty
\]
as $y \lesssim|\Omega|^{\frac{1}{n}}$ and $y \rightarrow \infty$. 
Next we claim that
\[
\frac{\Lambda}{2} J(v) \geq \log F(y)
\]
since $|\Omega|^{\frac{1}{n}} \gtrsim y$, $J(v) \gtrsim y^{\frac{n+\beta_1}{n+\beta_2}}$ for $y=v(0)$ sufficiently large.
On the other hand, from structural condition \refHc\ part (5) we have
\[
\log(F(y)) \leq C y^{\left(\frac{n+\beta_1}{n+\beta_2}\right)}
\]
for some constant $C$.
Thus, there is a universal $\Lambda_0$ such that, for $\Lambda > \Lambda_0$, $v(0)$ is bounded from above by a constant depending only on $A$ and universal data.
As before, if \refsHc\ holds, we may take $\Lambda_0=0$.
\\

\noindent {\bf Step 2: $|\Omega|$ is bounded above, and $v(0)$ is bounded below.}
\\

\noindent As before, we use $y=v(0)$. From Step 1, we can assume that $y \leq C$ for some universal constant $C$. Let $t=F(y)|\Omega|$. 
First assume \refHb\ holds. 
By Proposition~\ref{prop: G-bound} $(i)$ and $(ii)$, we have
\[
\mathcal{E}(v) \geq\begin{cases} -\log(t) + G^{-1}\left(C \frac{ F(y)G(\epsilon y)y^n}{t} -C\frac{F(y)y^{-\nu}}{t} + C^{-1}G(n\rho^+\epsilon_2D)\right) & \text{ if } \nu >0,\\
 -\log(t) + \Lambda G^{-1}(C t^{-1} \frac{F(y)G(\epsilon y)y^n}{t} -C\frac{F(y)(|\log(y)|+1)}{t} +C^{-1}G(n\rho^+\epsilon_2D)& \text{ if } \nu =0.
\end{cases}
\]
By structural condition \refHb\ parts $(2)$ and $(4)$, $F(y)G(\epsilon y)y^n$ and $ F(y)|\log(y)|$ are bounded for $y \leq C$. 
Furthermore, since $J(v) \geq 0$,  $t$ is bounded from below depending only on $A$. 
Thus, for $|\Omega|$ sufficiently large we have
\[
\mathcal{E}(v) \geq -\log(t) + \Lambda G^{-1}(C^{-1}G(n\rho^+\epsilon_2D)))
\]
since $G(s) \rightarrow +\infty$ as $s\rightarrow \infty$.
If $t$ is bounded, then $D$ is bounded, and hence $|\Omega|$ is bounded.
If $t$ is unbounded, then $ t \lesssim |\Omega| \lesssim D^n$ thanks to the upper bound for $v(0)$ established in Step 1.
Thus,
\[
\mathcal{E}(v) \geq -n\log D + \Lambda G^{-1}(C^{-1} G(n\rho^+\epsilon_2D))- C \rightarrow +\infty
 \]
 for $\Lambda$ sufficiently large by structural condition \refHb\ part $(4)$, (or $\Lambda >0$ in case \refsHb\ holds).
 In either case, we conclude that, for $\Lambda$ sufficiently large, $D$, and hence $|\Omega|$, is bounded depending only $\Lambda, A$, and universal data.
 Since $t$ is bounded from below, it follows that $v(0) \geq \delta_* >0$ is bounded from below depending only on universal data and $A$. 
 \\

\noindent Now we consider \refHc. In this case, Proposition~\ref{prop: G-bound} $(v)$ and $(vi)$ yield
 \[
\mathcal{E}(v) \geq\begin{cases} -\log(t) + G^{-1}\left(C \frac{ F(y)G(\epsilon y)y^n}{t} -C\frac{F(y)y^{-\nu}}{t} - \frac{C}{D}\right) & \text{ if } \nu >0,\\
 -\log(t) + \Lambda G^{-1}\left(C  \frac{F(y)G(\epsilon y)y^n}{t} -C\frac{F(y)(|\log(y)|+1)}{t} -\frac{C}{D}\right)& \text{ if } \nu =0.
\end{cases}
\]
Suppose that $t\rightarrow \infty$; then since $F(y)G(\epsilon y)y^n \gtrsim -y^{\gamma-\nu}$ for $\gamma>\nu$ by \refHc\, parts $(2)$ and $(4)$, and this is bounded below for $y<C$, we have
\[
\mathcal{E}(v) \geq  -\log(t) + \Lambda \epsilon (D^{-1} + t^{-1})^{-1/(n+\nu-1)} 
\]
by \refHc\ part (2).
Since $t= F(y)|\Omega| \lesssim D^n$, the result follows from \refHc\ part $(5)$, (or from \refsHc).
Thus we are reduced to the case that $t$ is bounded above and below.
Hence, $y=v(0)\rightarrow 0$ and $|\Omega| \rightarrow +\infty$.
Applying $F(s)\sim s^{\gamma}$, for $\gamma >\nu$ from \refHc\ part (4), we get
\[
\mathcal{E}(v) \geq -\log t + \Lambda G^{-1}( -Cy^{\gamma-\nu} - \frac{C^{-1}}{D}) \rightarrow +\infty
\]
using \refHc\ part (1) and $|\Omega| \lesssim D^n$.
\\

\noindent {\bf Step 3: ${\rm diam}(\Omega)$ is bounded above, and $\mathcal{E}(v) > -C$.}
\\

\noindent
We have proven that if either of the structure conditions \refHb\ or \refHc\ holds, then there are constants $\delta_*, \hat{D}$ such that, if $v \in \mathcal{C}_P$ satisfies $\mathcal{E}_{\Lambda}(v) \leq A$, for $\Lambda > \Lambda_0$, then $\delta_* \leq v(0) \lesssim |\Omega|^{\frac{1}{n}} \lesssim \hat{D}$.
We can easily upgrade the latter bound to a bound on ${\rm diam}(\Omega)$.
Recall that by~\eqref{eq: infuControlsSmallEcc}, we have
\[
-\inf u = v(0) \lesssim L_1
\]
where $L_1$ is the smallest eccentricity of the John ellipsoid of $\Omega$.
It follows that
\[
\delta_*^{n-1}{\rm diam}(\Omega) \leq |\Omega|,
\]
and hence ${\rm diam}(\Omega)\leq D^*$ for a constant $D^*$ depending only on universal data, $\Lambda$, and $A$.
Finally, the lower bound for the energy follows immediately, since $\mathcal{E}_{\Lambda}(v)$ is bounded below depending only on $v(0)$ and ${\rm diam}(\Omega)$.
\end{proof}

\section{Proof of the main theorems}\label{sec4:Proof}
We are now ready to prove the main theorems. We first observe that the boundary condition $v^\star = 0$ naturally arises along minimizing sequences in $\cC_P$.

\begin{lem}\label{lem:MAMasstoBoundary}
    For $v \in \cC_P$, there exists a $\tilde{v} \in \cC_P$ such that 
    \begin{enumerate}
        \item[(i)] $I(v) = I(\tilde{v})$,
        \item[(ii)] $J(\tilde{v}) \leq J(v)$, and
        \item[(iii)] $\tilde{v} \geq \phi_{\tilde{\Omega}}$ in $P$ with $\tilde{v} = \phi_{\tilde{\Omega}}$ on $\p P$. 
    \end{enumerate}
\end{lem}
\begin{proof}
    This is just \cite[Lemma 4]{TristanFreidYau}, and the proof follows from the argument therein, using Lemma~\ref{lem: poleNorm} in place of \cite[Proposition 3.1]{TristanFreidYau}.
\end{proof}

Recall that in Lemma~\ref{lem:JohnsOmegaBalanced} we proved that, when $\Omega= \{v^*<0\}$ is in John's position, then $v(0)\sim \inf_Pv$.  
In the next proposition, we prove a version of this result assuming instead that $v\in \mathcal{C}_P$ and $\mathcal{E}_{\Lambda}(v) \leq A$.  This can be viewed as an effective version of Lemma~\ref{lem:interiorNormalizationPole} saying that, under these assumptions, the minimizer of $\Omega \ni x\mapsto G(J(v-\langle x, \cdot \rangle)$ does not get too close to $\del \Omega$.

\begin{prop}\label{prop:effectiveNormalizationAll}
    Let $\mathtt{v}_o \geq 0$ be the vanishing order of $h$. Suppose that $(F,G)$ satisfy either \refHb${}_{\nu,\gamma}$\ or \refHc${}_{\nu,\gamma,\vec{\beta}}$\ for $\nu \geq \mathtt{v}_o$. Let $\Lambda_0$ be the constant from Proposition~\ref{prop:EEhatcomparisonConcaveb}.
    Let $v\in \mathcal{C}_P$ satisfy
    \begin{itemize}
        \item[(i)] $\mathcal{E}_{\Lambda}(v) \leq A$ for some $\Lambda > \Lambda_0$ and
        \item[(ii)] $ v \geq \phi_{\Omega}$ in $P$ and $v = \phi_{\Omega}$ on $\del P$, where $\Omega= \{ v^* <0\}$.
    \end{itemize}
    Then, there is a uniform constant $c>0$, depending only on the structural data, $\Lambda$, and $A$ such that $\inf_P v \geq c v(0)$.
\end{prop}
\begin{proof}
We argue by contradiction.
Suppose there is a sequence $v_k \in \mathcal{C}_P$ satisfying $\mathcal{E}_{\Lambda}(v) \leq A$, but such that
\[
\inf_P v_k = v_k(y_k) \leq k^{-1}v_k(0).
\]
By assumption $(i)$ and Proposition~\ref{prop:EEhatcomparisonConcaveb}, there are uniform constants $\delta_*, D^*$ such that 
\[
0<\delta^* < v_k(0) \lesssim {\rm diam}(\Omega_k) < D^*.
\]
Furthermore, since $v_k \geq \phi_{\Omega_k}$, and $v_k = \phi_{\Omega_k}$ on $\del P$, we conclude that $\nabla v_k(P) = \Omega_k$, and hence the $v_k$ are uniformly bounded in $C^1(\overline{P})$. 
We can therefore take a uniform limit $v_{\infty}=\lim_{k\rightarrow \infty}v_k$, and let $y_{\infty} = \lim_{k\rightarrow \infty}y_k$ with $y_{\infty} \in \ol{P}$. 
By assumption, $v_{\infty}(y_{\infty})=0$. 
It follows from Lemma~\ref{lem: poleNorm} that
    \[
    \int_{P}G(v_{\infty})\,dy^h = -\infty,
    \]
    so we apply Fatou's lemma to conclude that
    \[
    -\infty = \int_{P}G(v_{\infty})\, dy^h \geq \limsup_{k\rightarrow \infty}\int G(v_k)\, dy^h.
    \]
    On the other hand, consider $\widetilde{v}_k = v_k - \langle x_k,y\rangle$ such that $\widetilde{\Omega}_k = \Omega_k - x_k$ is in John's position. 
    Applying Lemma~\ref{lem:JohnsOmegaBalanced}, Proposition~\ref{prop: G-bound}, and Proposition~\ref{prop:EEhatcomparisonConcaveb} shows
    \[
    \int_{P} G(\tilde{v}_k)\, dy^h > -C
    \]
    for a uniform constant $C$ depending only on $\delta_*, D^*$ and structural data. 
    However, since $v_k \in \mathcal{C}_{P}$, we have $\int_{P}G(v_k)\, dy^h \geq \int_{P} G(\tilde{v}_k)\, dy^h >-C$, a contradiction.
\end{proof}

We can now take limits of a minimizing sequence for the energy $\mathcal{E}_{\Lambda}$ equation~\eqref{eqn:FreeBdyEqn}.
\begin{prop}\label{prop:MinimizersInCPUA}
    Under the assumptions of either Theorem~\ref{thm:Barycenter} or Theorem~\ref{thm: mainH23}, there exists $\Lambda_0 \geq 0$, depending only on structural data, such that, for all $\Lambda \geq \Lambda_0$ there is a function function $v \in \cC_P$ satisfying
    \[
    \mathcal{E}_{\Lambda}(v) = \inf_{\cC_{P}}\mathcal{E}_{\Lambda}.
    \]
    If \refsHb\ or \refsHc\ holds, then we may take $\Lambda_0=0$.
\end{prop}
\begin{proof}
    From Propositions~\ref{prop:ConvexGELowerBound} and~\ref{prop:EEhatcomparisonConcaveb}, there is a minimizing sequence $v_k \in \cC_P$ satisfying $\delta_* \leq v_k(0) \lesssim {\rm diam}(\Omega_k) \leq D^*$. 
    Furthermore, by Lemma~\ref{lem:MAMasstoBoundary}, we may assume that $v_k \geq \phi_{\Omega_k}$ and $v=\phi_{\Omega_k}$ on $\del P$.
    Let $u_k = v_k^*$.
    Then, by Proposition~\ref{prop:effectiveNormalizationAll}, we have
    \[
   D^* > v_k(0) = -\inf u_k \geq -u_k(0) = \inf v_k \geq cv_k(0) > c\delta_*
    \]
    for a uniform constant $c>0$.
    It follows that $u_k$ and $v_k$ are both uniformly bounded in $C^1$.
    We can therefore take uniform limits $u_k \rightarrow u_{\infty}$, and $v_k \rightarrow v_{\infty} \in \mathcal{C}_P$, and clearly $v_{\infty}^* = u_{\infty}$.
    By the uniform convergence, we have
    \[
    \mathcal{E}_{\Lambda}(v_{\infty}) = \inf_{v \in \mathcal{C}_P} \mathcal{E}_{\Lambda}(v).
    \] 
\end{proof}

To prove that $v$ solves equation~\eqref{eqn:FreeBdyEqn}, we show that $\nabla v$ solves an equivalent optimal transport problem. 
The following lemma, due to \cite[Lemma 5]{TristanFreidYau} will be used to show that $\nabla v$ is indeed a global minimizer along possible transport maps. 

\begin{lem}\label{lem:vtildeReplacement}
        Suppose that either of conditions \refHb${}_{\nu,\gamma}$\ or \refHc${}_{\nu,\gamma, \vec{\beta}}$\ are in effect, for $\nu \geq \mathtt{v}_o(h)$. Fix a constant $\lambda \in \mathbb{R}_{>0}$. For any function $\hat{v} \in \mathcal{C}^+$ there exists an affine function $\ell(y) = t_0 + \la x_0,y\rg$ such that by setting $\tilde{v}(y) = \hat{v}(y) + \ell(y)$, we have 
        \begin{enumerate}
            \item[(i)] $\tilde{v} > 0$, $\tilde{v} \in \cC_P$ and
            \item[(ii)] $J(\tilde{v}) = \lambda$.
        \end{enumerate}
    \end{lem}
\begin{proof}
Let $\hat{u}$ be the Legendre transform of $\hat{v}$ and $\Omega_t = \{\hat{u}<t\}$.
For any $t > \inf \hat{u} = -\hat{v}(0)$, we define $\tilde{v}_t(y) = \hat{v}(y) + t - \la y , x_t\rg$ where we choose $x_t \in \Omega_t$ from Lemma~\ref{lem:interiorNormalizationPole}.
This process is continuous since $G$ and $G^{-1}$ are continuous and $h > 0$ on $P$.
Then, we have
\[
    J(\tilde{v}_t) \geq J(\hat{v}+t) =G^{-1}\left(\frac{1}{H}\int_{P} G(\hat{v}+t)\,dy^h\right).
    \]
    Since $\hat{v} \geq 0$ and $G$ is increasing, 
    \[
    G^{-1}\left(\frac{1}{H}\int_{P} G(\hat{v}+t)\, dy^h\right) \geq G^{-1}(G(t))=t
    \]
and therefore $J(\tilde{v}_t) \rightarrow +\infty$ as $t\rightarrow +\infty$.
On the other hand, as $t \rightarrow \inf u$, we have 
\[
\int G(\tilde{v}_t)\,dy^h \rightarrow -\infty
\]
by Lemma~\ref{lem: poleNorm}. 
Therefore, as $t \rightarrow \inf \hat{u}$, we have $J(\tilde{v_t}) \rightarrow 0$, and the result follows by the intermediate value theorem.
\end{proof}

\begin{lem}\label{lem:vtildeReplacementH1}
Suppose that \refHa\ holds. 
Fix a constant $\lambda \in \mathbb{R}_{>0}$. 
For any function $\hat{v} \in \mathcal{C}^+$, there exists a $t \in \mathbb{R}$ so that $\tilde{v}(y) = \hat{v}(y) +t$ satisfies
        \begin{enumerate}
            \item[(i)] $\tilde{v} > 0$ and
            \item[(ii)] $I(\tilde{v}) = \lambda$.
        \end{enumerate}
    \end{lem}
\begin{proof}
Consider the family $\tilde{v}_t = \hat{v}+t$ for $t \in (\inf \hat{u}, \infty)$.
Let $\tilde{u}_t = \tilde{v}_t^*$, so $\tilde{u}_t = \hat{u}-t$ and 
\[
I(\tilde{u}_t) = \int_{\{\hat{u}<t\}}F(t-\hat{u})\, dx.
\]
Since $F$ is convex, $I(\tilde{u}_t)\rightarrow +\infty$ as $t\rightarrow +\infty$. 
Similarly, as $t\rightarrow \inf \hat{u}$, we have $I(\tilde{u}_t) \rightarrow 0$.
\end{proof}

We can now prove Theorems~\ref{thm:Barycenter} and~\ref{thm: mainH23} by showing that $\nabla u$ solves an optimal transport problem equivalent to equation~\eqref{eqn:FreeBdyEqn}. A similar argument was used by \cite{Klartag, TristanFreidYau}.
\begin{prop}\label{prop:MinIsSln} 
The function $v$ constructed in Proposition~\ref{prop:MinimizersInCPUA} solves equation~\eqref{eqn:FreeBdyEqnLegendre} and has $v \in C^{1,\alpha}(\ol{P})$.
\end{prop}

\begin{proof}
    Let $u$ be the Legendre transform of $v$, the minimizer of $\cE_{\Lambda}$ over $\mathcal{C}_P$, as constructed in Proposition~\ref{prop:MinimizersInCPUA}. 
    We define two probability measures
    \[
    d\mu = \frac{f(u(x))\chi_\Omega\, dx}{\int_\Omega f(u(x))\, dx}\quad \text{and}\quad d\nu = \frac{g(v(y)) \, dy^h}{\int_P g(v(y)) \, dy^h}.
    \]
    Our goal is to show that for any convex function $\hat{v} : \ol{P} \to \bR$ with $\hat{u}$ its Legendre transform, we have 
    \begin{equation}\label{ineq:optimalTransport}
    \int_{\mathbb{R}^n} u(x)\, d\mu + \int_P v(y)\, d\nu \leq \int_{\mathbb{R}^n} \hat{u}(x)\, d\mu + \int_P \hat{v}(y)\, d\nu.
    \end{equation}
    This inequality implies that $(v,u)$ solves the optimal transport problem for the measures $d\mu, d\nu$.
    
   We first consider the case when $(F,G)$ satisfy either \refHb${}_{\nu,\gamma}$ or \refHc${}_{\nu,\gamma,\vec{\beta}}$\ for $\nu \geq \mathtt{v}_o$.
   By Lemma~\ref{lem:vtildeReplacement}, we can choose $x_0, t_0$ so that $\tilde{v}= \hat{v} +t_0 -\langle x_0,y\rangle$ satisfies $\tilde{v} \in \cC_P$ and $J(\tilde{v})=J(v)$. 
   Let $\tilde{u}= \tilde{v}^*$.
   Since $v$ minimizes $\mathcal{E}_{\Lambda}$ for $\cC_P$, we have
    \[
    -\log I(u) + \Lambda J(v) \leq -\log I(\tilde{u}) + \Lambda J(\tilde{v})
    \]
    and since $J(\tilde{v}) = J(v)$, we conclude that $I(u) \geq I(\tilde{u})$.
    By convexity of the maps $s\mapsto F(\max\{-s, 0\})$ and $F$, we have
    \[
    f((-u)_+)(u-\tilde{u}) \leq F((-\tilde{u})_+)-F((-u)_+).
    \]
    Integrating over $\mathbb{R}^n$ shows
    \begin{equation}\label{eq: uIneqOT}
    \int_{\mathbb{R}^n}(u-\tilde{u})\,d\mu \leq  \frac{I(\tilde{u}) - I(u)}{{\int_\Omega f(-u(x))\, dx}} = \frac{I(\tilde{u}) - I(u)}{{\int_\Omega f(-u(x))\, dx}} \leq 0,
    \end{equation}
    where we used the translation invariance of $I$. On the other hand, since $d\mu$ is a probability measure, we obtain
    \[
    \int_{\mathbb{R}^n}(u-\hat{u})\,d\mu + t_0 \leq  0.
    \]
    Similarly, the concavity of $G$ yields
    \[
    g(v) (\tilde{v}-v) \geq G(\tilde{v})-G(v) .
    \]
    Integrating over $P$ with respect to $dy^h$ yields
    \begin{equation}\label{eq: vIneqOT}
    0 \leq \int_{P} (\tilde{v}-v)\, d\nu  = \int_{P} (\hat{v}-v)\, d\nu + t_0,
    \end{equation}
    where we used that $d\nu$ has barycenter at the origin, since $v\in \cC_P$. 
    Combining~\eqref{eq: uIneqOT} and~\eqref{eq: vIneqOT} yields~\eqref{ineq:optimalTransport}.

    Now suppose that $(F,G)$ satisfy \refHa. Clearly~\eqref{ineq:optimalTransport} is invariant under adding constants to $\hat{v}$, so we may assume that $\hat{v}>0$, and, by Lemma~\ref{lem:vtildeReplacementH1} that $I(\hat{v})=I(v)$. 
    Since $F$ is convex and $I(u) = I(\hat{u})$,~\eqref{eq: uIneqOT} holds. 
    Finally, because $v$ minimizes $\mathcal{E}_{\Lambda}$ over $\cC^+$ and $I(u)=I(\hat{u})$, we have
    \[
    \int_{P}v\,dy^h \leq \int_{P}\hat{v}\,dy^h,
    \]
    and~\eqref{ineq:optimalTransport} follows.
    
    We can now conclude that $\nabla u$ is the solution to the optimal transport problem from measure $d\mu$ to $d\nu$ on $\bR^n$ from Gangbo-McCann \cite{Gangbo-McCann}. 
    Therefore, $u$ solves equation~\eqref{eqn:FreeBdyEqn} in the Alexandrov sense \cite{Caffarelli}.
    Because $F'(s)= f(s)\sim s^{\gamma -1}$ for $\gamma \geq 1$ and $u$ vanishes to order $1$ on the $\del\Omega$, it follows from \cite[Lemma 2.6]{Jhaveri-Savin} that $d\mu$ is a doubling measure.
    Similarly, since $h(y)dy$ is doubling by assumption, and $g(v(y))$ is bounded above and below, it follows that $d\nu$ is doubling. 
    Therefore, by \cite[Theorem 1.1]{Jhaveri-Savin} $u \in C^{1,\alpha}(\overline{\Omega})$ and $v\in C^{1,\alpha}(\ol{P})$.
\end{proof}
Theorems~\ref{thm:Barycenter} and~\ref{thm: mainH23} follow as consequences of Propositions~\ref{prop:MinimizersInCPUA} and \ref{prop:MinIsSln}.
We can also deduce higher interior regularity given sufficiently regular $f,g,h$ from Caffarelli's regularity theory \cite{Caffarelli, Caffarelli2, Caffarelli3}, showing that for $f,g,h \in C^{k,\beta}$, we have $v \in C^{k+2, \alpha}(P)$ for $\alpha < \beta$. 

\begin{lem}\label{lem: FBStrictConvex}
   Assume there is a constant $C>0$ such that $C^{-1}\leq h \leq C$ on $\overline{P}$, and suppose that $u$ solves~\eqref{eqn:FreeBdyEqn}. Then $\Omega= \{u<0\}$ is strictly convex, in the sense that $\del \Omega$ does not contain a line segment.
\end{lem}
\begin{proof}
The argument is based on work of Savin \cite{SavinObs} and Caffarelli \cite{Caffarelli4}. Since $f(s) \sim s^{\gamma-1}$ for some $\gamma \geq 1$, and $u$ vanishes to order $1$ on $\del \Omega$, the measure $\frac{f(-u)}{g(u^{\star}) h(\nabla u)} dx$ is doubling by \cite[Lemma 2.6]{Jhaveri-Savin}. The result now follows from the arguments of Savin \cite{SavinObs}; see also \cite[Proposition 5.3]{TristanFreidYau}.
\end{proof}

\begin{ques}
    It would be very interesting to understand the regularity of the free boundary.  For the Monge-Amp\`ere obstacle problem, the regularity of the free boundary was established in \cite{Huang-Tang-Wang}.
\end{ques}

\section{Uniqueness, Examples and Applications}\label{sec5:ExamplesApplications}

In this section we discuss the uniqueness of solutions to the free boundary equation~\eqref{eqn:FreeBdyEqn}, as well as many natural settings where Theorem~\ref{thm:Barycenter} and Theorem~\ref{thm: mainH23} apply.

\subsection{Uniqueness}
When Properties \refHb\ or \refHc\ are in effect, the energy functional $\mathcal{E}_{\Lambda}$ is not convex and solutions of~\eqref{eqn:FreeBdyEqn} are not local minimizers of $\mathcal{E}_{\Lambda}$, so proving uniqueness seems rather subtle.
However, in case \refHa\ holds, $\mathcal{E}_{\Lambda}$ is convex, at least when $F$ is log concave, as explained in Claim~\ref{claim:convexityOfI} and Claim~\ref{claim:convexityOfJ}. 
We begin by discussing the uniqueness of our solutions.
We prove the following conditional uniqueness result:
\begin{lem}\label{lem: conditionalUniqueness}
    Suppose that $(F,G)$ satisfy \refHa, $F$ is log-concave, and $P$ has barycenter at the origin. Let $v_0,v_1$ be minimizers of $\mathcal{E}_{\Lambda}$, and let $u_0, u_1$ be their Legendre duals. 
    Then,
    \begin{itemize}
        \item[(i)] there is a $\mathtt{w}\in\{u_0<0\}$ such that $\{u_1<0\} = \{u_0<0\}-\mathtt{w}$, and
        \item[(ii)] if $J(v_0)=J(v_1)$, or $I(u_0)=I(u_1)$, then $u_0(x)=u_1(x-\mathtt{w})$.
    \end{itemize}
\end{lem} 

\begin{proof}
    Let $v_t= (1-t)v_0+tv_1$, and let $u_t = v_t^*$. Since $t\rightarrow \mathcal{E}_{\Lambda}(v_t)$ is convex, and minimized at $t=0,1$, we have that $\mathcal{E}(v_t)$ is constant. 
    Clearly
    \[
    J(v_t) = (1-t)J(v_0) + t J(v_1)
    \]
    since $G(s)=s$ by assumption.
    Thus, $t \mapsto -\log(I(u_t))$ is linear.
    Define
    \[
    \Phi(x,t) = \begin{cases} -\log(F(-u_t(x)) & \text{ if } x\in \{u_t <0\},\\
    +\infty & \text{ else}.
    \end{cases}
    \]
    By the equality case of Prekopa's theorem, we have
    \begin{equation}\label{eq: PrekopaEquality}
    \Phi(x,t) = \Phi(x+t\mathtt{w},0)-tc
    \end{equation}
    for some $\mathtt{w} \in \mathbb{R}^n$, $c\in \mathbb{R}$. It follows immediately that
    \[
    \{u_t<0\}+t\mathtt{w} = \{u_0<0\}.
    \]
    Since $0 \in \{u_1<0\} \cap \{u_0 <0\}$, we deduce assertion $(i)$. 
    Now assume that $J(v_0)= J(v_1)$, (resp.~$I(v_0)=I(v_1)$), and hence also $I(v_0)=I(v_1)$, (resp.~$J(v_0)=J(v_1)$), since both $v_0, v_1$ are minimizers of $\mathcal{E}_{\Lambda}$.
    Then, $c=0$ in~\eqref{eq: PrekopaEquality} and $u_0(x)=u_1(x-w)$.
\end{proof}

We note that $ s \mapsto f(-s)$ is decreasing when $f$ is increasing, so even though the free boundary is known to be unique (up to translations) by Lemma~\ref{lem: conditionalUniqueness}, uniqueness of the solutions does not follow from the maximum principle. 
In the next two lemmas, we provide conditions under which assumption $(ii)$ of Lemma~\ref{lem: conditionalUniqueness} holds, implying uniqueness in these settings.

\begin{lem}\label{lem: homogenousJFix}
    Suppose that $(F,G)$ are homogeneous of degree $(k,l)$. Then, at any $v \in \mathcal{C}_P$ satisfying $\mathcal{E}_{\Lambda}(v) = \inf_{\cC_P} \mathcal{E}_{\Lambda} = \mathcal{E}_{*}$, we have
    \[
    J(v) = n+k \quad \text{ and }\quad  -\log I(v) = \mathcal{E}_*-\Lambda(n+k).
    \]
\end{lem}
\begin{proof}
    Consider the variation $v_{t}(y) = tv(y)$ for $t \in (1-\epsilon, 1+\epsilon)$. 
    Since $G$, and hence $g$, are homogeneous, we see that if $v \in \cC_P$, then so is $v_{t}$. Direct computation shows that $u_{t} = v_{t}^*$ is given by the formula $u_{t}(x) = t u^*(t^{-1}x)$, so in particular we have
    \[
    \begin{aligned}
    J(v_{t}) &= t J(v),\\
    I(u_{t}) &= t^{n+k}I(u).
    \end{aligned}
    \]
    Then, since $\frac{d}{dt}\big|_{t =1}\mathcal{E}_{\Lambda}(v_t)=0$, we have $J(v) = n+k$, and the result follows.
\end{proof}

\begin{lem}\label{lem: exponentialJFix}
    Suppose that $F(s)=a^{-1}(e^{as}-1)$ for $a>0$ and $G(s)=s$; so in particular, \refHa\ holds. 
    Further suppose that $u$ solves
    \[
    \begin{split}
    \det D^2 u &= \frac{H}{I(v)\Lambda} e^{-au} \chi_{\{u<0\}},\\
    \nabla u(\mathbb{R}^n) &= P.
    \end{split}
    \]
    Then,
    \[
    I(v) = \frac{1}{|\{u<0\}|}\frac{H}{(\Lambda|P|-H a)}.
    \]
\end{lem}
\begin{proof}
    Integrating the equation yields
    \[
    \begin{aligned}
    |P| &= \frac{H}{\Lambda I(v)} \int_{\{u<0\}}e^{-au}\\
    &= \frac{H}{I(v)\Lambda }a \int_{\{u<0\}} a^{-1}(e^{-au}-1) + \frac{H}{\Lambda I(v)}|\Omega|\\
    &= \frac{H}{\Lambda}\left( a+ \frac{|\Omega|}{I(v)}\right),
    \end{aligned}
    \]
    and the result follows.
\end{proof}

\begin{cor}\label{cor: uniquenessForExamples}
    Suppose that $P$ has $h$-barycenter at the origin and either of the following conditions hold:
    \begin{itemize}
        \item[(i)] $(F, G)$ satisfy \refHa and $F$ is log-concave and homogeneous of degree $k$, or
        \item[(ii)] $F(s)= a^{-1}(e^{as}-1)$ where $a>0$, and $G(s)=s$.
    \end{itemize}
    Then, the solution to~\eqref{eqn:FreeBdyEqn}, if it exists, is unique.
\end{cor}
\begin{proof}
    Under assumption $(i)$, the result follows from Lemma~\ref{lem: conditionalUniqueness} and Lemma~\ref{lem: homogenousJFix}. 
    Now assume condition $(ii)$ holds. 
    By Lemma~\ref{lem: conditionalUniqueness}, part $(i)$, we deduce that, for any two solutions $u_0, u_1$ of~\eqref{eqn:FreeBdyEqn}, we have $|\{u_0<0\}| = |\{u_1<0\}|$, and hence by Lemma~\ref{lem: exponentialJFix}, we conclude that $I(u_0)=I(u_1)$, and hence uniqueness follows from Lemma~\ref{lem: conditionalUniqueness} part $(ii)$.
\end{proof}

\subsection{Examples and applications}
In this section we discuss some examples and applications of our results.
\\

\subsubsection{The Monge-Amp\`ere eigenvalue}

The Monge-Amp\`ere eigenvalue problem was introduced by Lions \cite{Lions}. 
Given a bounded convex domain $\Omega\subset \mathbb{R}^n$, we seek a constant $\lambda>0$ and a convex function $u : \Omega \rightarrow \mathbb{R}$ such that
\begin{equation}\label{eq: MAEV}
\begin{split}
\det D^2 u &= \lambda (-u)^n\;\text{ in } \Omega,\\
u\big|_{\del \Omega} &= 0.
\end{split}
\end{equation}
Lions proved that there exists a unique $\lambda= \lambda(\Omega)$ and a unique, up to scaling, convex function $u$ solving~\eqref{eq: MAEV}. 
Tso \cite{Tso} provided a variational characterization of the eigenvalue.
The boundary regularity of the Monge-Amp\`ere eigenvalue was obtained by Le-Savin \cite{Le-Savin}.
The Monge-Amp\`ere eigenvalue is known to have interesting concavity properties, analogous to the Brunn-Minkowski theory \cite{Salani}.
Many extensions and generalizations of this circle of ideas have been studied in the literature (see e.g.,~\cite{Tong-Yau} and the references therein).
Theorem~\ref{thm:Barycenter} yields the following:

\begin{theorem}\label{thm: MAEV}
    Suppose that $P\subset \mathbb{R}^n$ is a bounded convex body with barycenter the origin. 
    Then, for any $\lambda >0$ there exists a unique convex domain $\Omega_{\lambda}$, with barycenter at the origin, and a unique convex function $u_{\lambda}:\Omega_{\lambda} \rightarrow \mathbb{R}$ such that
    \begin{equation}\label{eq: fbdryMAEV}
    \begin{split}
    \det D^2u_{\lambda} &= \lambda (-u_{\lambda})^n  \; \text{ in } \Omega_{\lambda},\\
    u_{\lambda}\big|_{\del \Omega_{\lambda}} &=0,\\
    \nabla u_{\lambda}(\Omega_{\lambda}) &= P.\\
    \end{split}
    \end{equation}
\end{theorem}

\begin{proof}
    We apply Theorem~\ref{thm:Barycenter} with $F(s) = \frac{s^{n+1}}{n+1}$, $G(s)=s$, and density $h=1$.
    Since $\log(s) =o(s)$ as $s\rightarrow \infty$, we are in the setting of \refsHa, and hence for any $\Lambda >0$, we can solve
    \begin{equation}\label{eq: FreeBdryMAEV}
    \begin{split}
    \det D^2u_{\Lambda} &= \lambda(\Lambda) \frac{|P|}{\Lambda I(u)} (-u)^{n} \chi_{\{u<0\}},\\
    \nabla u_{\Lambda}(\mathbb{R}^n) &= P.
    \end{split}
    \end{equation}
    By Corollary~\ref{cor: uniquenessForExamples}, $u_{\Lambda}$ is unique if we fix the barycenter of $\{u_{\Lambda}<0\}$ to be the origin.
    Finally, we observe that by rescaling $u_{\Lambda} \mapsto t^{-1}u_{\Lambda}(tx)$, we can achieve any constant on the right hand side of~\eqref{eq: FreeBdryMAEV}.
\end{proof}

Of course, our results also apply to solve the equations
    \begin{equation}\label{eqn:genEigenvalueEqn}
        \begin{split}
            \det D^2 u &= \lambda (-u)^k \;\text{ in }\Omega,\\
            u\big\vert_{\p \Omega} &=0, \\
            \nabla u(\Omega) &= P\\
        \end{split}
    \end{equation}
    for any $k \geq 0$. Since the solution to~\eqref{eqn:genEigenvalueEqn} is unique up to translation, we obtain ``reconstruction" type theorems, of the following sort:

    \begin{cor}\label{cor: reconstruction}
        If $P$ is a bounded convex set with barycenter at the origin, then there exist a unique convex set $\Omega$ and unique convex function $u: \Omega \rightarrow \mathbb{R}$ such that $\Omega$ has barycenter at the origin, and
        \[
        \begin{split}
            \det D^2 u &= 1\; \text{ in }\Omega,\\
            u\big\vert_{\p\Omega} &=0, \\
            \nabla u(\Omega) &= P.\\
        \end{split}
        \]
        That is, the pair $(u,\Omega)$ can be reconstructed from $P$. Furthermore, $\del \Omega$ is strictly convex and $C^{1,\alpha}$ for some $\alpha >0$. 
    \end{cor}
    \begin{proof}
        The existence and uniqueness of $(u,\Omega)$ follow from Theorem~\ref{thm:Barycenter} and Corollary~\ref{cor: uniquenessForExamples}. That $\del \Omega$ is strictly convex follows from Lemma~\ref{lem: FBStrictConvex}. That $\del\Omega$ is $C^{1,\alpha}$ follows from the implicit function theorem, since $\del\Omega= \{u=0\}$, and $\nabla u: \del\Omega \rightarrow \del P$, which is disjoint from $0$.
    \end{proof}

    \begin{remark}
        It would be interesting to know whether Corollary~\ref{cor: reconstruction} continues to hold if $P$ is not convex.
    \end{remark}

We note the following interesting question, which arises from Theorem~\ref{eq: MAEV}. First, choose the unique $R= R(n)>0$ so that the solution of the Monge-Amp\`ere eigenvalue problem~\eqref{eq: MAEV} on $B_{R}(0)\subset \mathbb{R}^n$ satisfies
\[
\begin{split}
\det D^2 u &= \lambda (-u)^n\; \text{ in } B_{R}(0),\\
u\big|_{\del B_{R}(0)} &= 0,\\
\nabla u(B_{R}(0)) &= B_{R}(0).
\end{split}
\]

Fix a convex body $P_0$ with barycenter at the origin and apply Theorem~\ref{thm: MAEV} to find a function $u_0$ solving~\eqref{eq: fbdryMAEV} with $P=P_0$. 
Let $P_1= \{u_0<0\}$ and repeat the process.
In this way, we obtain a sequence of convex bodies $P_0\to P_1\to \cdots \to P_k\to \cdots$, each with barycenter at the origin.

\begin{ques}
What are can be said about this dynamical system?  
Are there fixed points other than $B_{R(n)}(0)$? 
\end{ques}

\subsubsection{Gauss curvature and \texorpdfstring{$L_p$}{Lp} affine surface measure}

In this section we consider Monge-Amp\`ere equations arising from geometric problems involving Gauss curvature. Recall that the classical Minkowski problem \cite{Lewy, Nirenberg, Cheng-Yau, Cheng-Yau-2, Pogorelov, Alexandroff, Caffarelli4, Caffarelli5} asks whether it is possible to prescribe the Gauss curvature of the boundary of a convex body. Recently there has also been a great deal of interest in $L_p$ extensions of the Minkowski problem (see e.g.,~\cite{HLYZ, LYZ, Chou-Wang}). Here we explain how our results imply the solution of a hemispherical version of the Minkowski problem.
One can similarly obtain hemispherical versions of the $L_p$-Minkowski problems for sufficiently large $p$ using Theorem~\ref{thm: mainH23}.

Let $P\subset \mathbb{R}^n$ be a bounded convex set with $0\in P$, and let $K(y): \overline{P} \rightarrow \mathbb{R}_{+}$. Suppose that
\begin{equation}\label{eqn:GaussCurvBarycenter}
\int_{P}  \vec{y}K(y)(1+|y|^2)^{\frac{n+2}{2}}\, dy=0.
\end{equation}
Let $\bS^{n}_+= \{x \in \bS^n: x_{n+1}>0\}$ and let $\mathtt{s}^{-1}:\mathbb{R}^n \rightarrow \bS^n_+$ be the inverse of the radial projection map
\[
\bS^n_+ \ni x = \mathtt{s}^{-1}(y)= \frac{(y,1)}{\sqrt{1+|y|^2}}.
\]
We regard $P \subset \mathbb{R}_y^n \subset \mathbb{R}_{(y',y_{n+1})}^{n+1}$ by the inclusion $y\mapsto (y,1)$ so that $P\subset \{y_{n+1}=1\}\subset\mathbb{R}^{n+1}$. 
On the free boundary side, let $\mathbb{R}^{n+1}_{+}= \{x \in \mathbb{R}^{n+1}_{x',x_{n+1}} : x_{n+1} \geq 0\}$.

\begin{theorem}\label{thm: prescribedGauss}
    Suppose that~\eqref{eqn:GaussCurvBarycenter} holds. 
    There exists a hemisphere $\Sigma\subset \mathbb{R}_+^{n+1}$, with $\del \Sigma \subset \{x_{n+1}=0\}$ satisfying 
    \begin{itemize}
        \item[(i)] $\nu_{\Sigma}(\Sigma) = \mathtt{s}^{-1}(P)$ and
        \item[(ii)] $K_{\Sigma}(x) = K(\mathtt{s}(\nu_{\Sigma}(x)))$,
    \end{itemize}
     where $\nu_{\Sigma}$ is the Gauss map of $\Sigma$ and $K_{\Sigma}$ the Gauss curvature.
\end{theorem}
\begin{proof}
    By Theorem~\ref{thm:Barycenter} there exists a bounded convex set $\Omega\subset \mathbb{R}^n$, and a function $u:\Omega \rightarrow \mathbb{R}$ solving the free boundary problem
    \[
    \begin{aligned}
        \det D^2u &= \lambda K(\nabla u) (1+|\nabla u|^2)^{\frac{n+2}{2}} \; \text{ in } \Omega,\\
        u \big|_{\del \Omega}&=0,\\
        \nabla u(\Omega) &=P
    \end{aligned}
    \]
    for some $\lambda >0$. 
    By rescaling $u(y) \to tu(t^{-1}y)$, we may assume that $\lambda=1$. 
    Let $\Sigma = \text{Graph}(-u) \subset \mathbb{R}^{n+1}_+$. Then $\del \Sigma = \del \Omega \subset \{x_{n+1}=0\}$.
    For $y\in \mathbb{R}^n$ we have
    \[
    \nu_{\Sigma}(y) = \frac{(\nabla u, 1)}{\sqrt{1+|\nabla u|^2}}
    \]
    is the upwards pointing normal vector.
    The Gauss curvature of $\Sigma$ is
    \[
    K_{\Sigma}(x) = \frac{\det D^2u}{(1+|\nabla u|^2)^{\frac{n+2}{2}}} = K(\nabla u(y)) = K(\mathtt{s}(\nu_{\Sigma}(x))
    \]
    as desired.
\end{proof}

We note that the hemisphere of Theorem~\ref{thm: prescribedGauss} is unique, up to translations, among all graphical hemispheres by Corollary~\ref{cor: uniquenessForExamples}. Furthermore,~\eqref{eqn:GaussCurvBarycenter} is a necessary condition for the existence of a graphical hemisphere.
Theorem~\ref{thm: prescribedGauss} can be given an intrinsic geometric formulation as follows.
Let $\widetilde{P} = \mathtt{s}^{-1}(P)$. 
We claim that this is a spherically convex set in $\bS^n_+$ containing the north pole.
Indeed, it is straight-forward to see that $\mathtt{s}$ and $\mathtt{s}^{-1}$ interchange the notions of convexity in the Euclidean and spherical senses.

To give an intrinsic formulation of~\eqref{eqn:GaussCurvBarycenter}, we note that the spherical metric is given in coordinates induced by $\mathtt{s}$ by $d\sigma(x) = (1 + |y|^2)^{-\frac{n+1}{2}} dy$, and so the measure 
\[
K(y)(1 + |y|^2)^\frac{n+2}{2}dy = \frac{K(x)}{x_{n+1}^{2n+3}}\,d\sigma(x).
\]
Since $\vec{y} = \tfrac{\vec{x}'}{x_{n+1}}$, the barycenter condition in equation~\eqref{eqn:GaussCurvBarycenter} is equivalent to 
\begin{equation}\label{eqn:SphericalBarycenter}
    \int_{\widetilde{P}}\vec{x'}\frac{K(x)}{x_{n+1}^{2(n+2)}}\,d\sigma(x) = 0,
\end{equation}
which says that the north pole is the barycenter of $\widetilde{P}$ with respect to this weighted measure.
This leads to the following intrinsic restatement of Theorem~\ref{thm: prescribedGauss}, which is a hemispherical version of the Minkowski problem:

\begin{theorem}\label{thm: prescribedGaussB}
    Let $P \Subset \bS^n_+$ be spherically convex. Suppose $K : \overline{P} \to \bR_+$ satisfies~\eqref{eqn:SphericalBarycenter}. 
    There exists a hemisphere $\Sigma\subset \mathbb{R}_+^{n+1}$ with $\del \Sigma$ contained in the hyperplane $\{x_{n+1}=0\}$ satisfying
    \begin{itemize}
        \item[(i)] $\nu_\Sigma(\Sigma) = P$ and
        \item[(ii)] $K_{\Sigma}(x) = K(\mathtt{s}(\nu_{\Sigma}(x)))$,
    \end{itemize}
    where $\nu_{\Sigma}$ is the Gauss map of $\Sigma$ and $K_{\Sigma}$ the Gauss curvature.
    \end{theorem}
    
The restriction~\eqref{eqn:SphericalBarycenter} can be removed in many examples as follows. We say that $\xi_0 \in \bS^n_+$ is the spherical $K$-barycenter if
\[
\int_{\widetilde{P}}\vec{\xi}\frac{K(\xi)}{\xi_{n+1}^{2(n+2)}}d\sigma(\xi) \quad\text{is parallel to}\quad \vec{\xi}_0.
\]
Suppose that ${\rm diam}(\widetilde{P}) < \tfrac{\pi}{2}$, so that a rotation of $\widetilde{P}$ taking any $\eta\in \overline{\widetilde{P}}$ to the north pole preserves the containment $\widetilde{P}\subset \bS^n_+$.
We claim that, after a rotation, condition~\eqref{eqn:SphericalBarycenter} can always be satisfied. 
To see this, for $\eta \in \widetilde{P}$, consider $R_{\eta} \in SO(n+1)$ that rotates $\eta$ to the north pole along the unique geodesic connecting them.
Let $\xi_\eta$ be the spherical $K$-barycenter of $\widetilde{P}$ rotated by $R_{\eta}$ and define $b_{\eta} = \mathtt{s}(\xi_\eta)$. 
For $\eta \in \p \widetilde{P}$, $\mathtt{s}(b_\eta) \ne 0$. 
Thus, the map $b : \p \widetilde{P} \to \bR^n \setminus \{0\}$ is homotopic to a degree-1 map between spheres.
Therefore, there must exist some $R_\eta$ such that $b_\eta = 0$; otherwise we could normalize $\tfrac{b_\eta}{\|b_\eta\|}$, which would be a retraction of $\widetilde{P}$ onto its boundary, a contradiction. 
Therefore if ${\rm diam}(\widetilde{P})<\frac{\pi}{2}$, then Theorem~\ref{thm: prescribedGaussB} can be applied to produce a hemisphere with Gauss curvature $K$ and boundary on a hyperplane.

\subsubsection{Optimal transport}

We now consider the existence of homogeneous optimal transport maps between a cone with compact cross-section and a half space. Such maps arise as blow-up limits of optimal transport maps between convex domains, equipped with degenerate measures. These maps are intimately related to the regularity of optimal transport, as explained in the recent work of the first author and Tong \cite{Collins-Tong}. 

Let $P \subset \mathbb{R}^{n}$ be a bounded convex set with $0\in P$. Consider the embedding $P \hookrightarrow \mathbb{R}^{n+1}$ induced by embedding $\mathbb{R}^{n} \hookrightarrow \mathbb{R}^{n+1}$ as the affine plane $\{y_{n+1}=1\}$. 
Let $\mathtt{C}(P)\subset \mathbb{R}^{n+1}$ denote the cone over $P$ with vertex at $0$. 
Let $h:P\rightarrow \mathbb{R}_{>0}$ be a $C^{\alpha}$ function, such that $h\sim d_{\del P}^{\alpha}$, and denote by $\rho^{\alpha}$ its degree-$\alpha$ homogeneous extension to $\mathtt{C}(P)$. 
We consider the optimal transport problem for a convex function $\phi: \mathtt{C}(P) \rightarrow \mathbb{R}$ satisfying
\begin{equation}\label{eq: OTCONES}
\begin{aligned}
\phi_{n+1}^{\beta} \det D^2 \phi &= \rho^{\alpha} \;  \text{ in } \mathtt{C}(P),\\
\nabla \phi (\mathtt{C}(P)) &= \mathtt{C}_{+}:=\{x \in \mathbb{R}^{n+1} : x_{n+1}>0\}.
\end{aligned}
\end{equation}
Such a function $\phi$ describes an optimal transport map between $\nabla \phi : (\mathtt{C}(P), \rho^{\alpha} dx) \rightarrow (\mathtt{C}_+, y_{n+1}^{\beta}dy)$.

\begin{lem}\label{lem: homogReductOT}
    Let $P\subset \mathbb{R}^n$ be a bounded convex set with $0\in P$. Suppose that $v$ solves the following Monge-Amp\`ere equation:
    \begin{equation}\label{eqn:simplexLinkEqn}  
    \begin{split}
    \det D^2 v &= \frac{h(y')}{v^{n+2+\alpha}(-v^\star)^\beta} \;\text{ in } P\subset \mathbb{R}^{n},\\
    v^\star\big\vert_{\p P} &= 0
    \end{split}
    \end{equation}
    where $h(y')\sim d(y',\del P)^{\alpha}$. If $y' \in \mathbb{R}^{n} = \{y_{n+1}=1\}\subset \mathbb{R}^{n+1}$, then, up to rescaling, $\varphi(y_{n+1} y', y_{n+1}) = (y_{n+1} v(y'))^{1 + \gamma}$ for $\gamma = \frac{n+1+\alpha}{n+1+\beta}$, solves~\eqref{eq: OTCONES}.
\end{lem}

\begin{proof}
This follows a similar computation to \cite[Proposition 5.1]{TristanFreidYau}.
Let $(y_1,\ldots,y_n,y_{n+1}) = (y', y_{n+1})$ where $v$ is defined on $P$. 
    We write that $\varphi(y) = \left(y_{n+1}v\right)^{1+\gamma}$.
    The density $\rho^{\alpha}(y,y') = y_{n+1}^\alpha h(\frac{y'}{y_{n+1}})$ is homogeneous of degree $\alpha$.
    Therefore, we write that $\varphi(y) = \left(y_{n+1}v\right)^{1+\gamma}$.
    The argument function $v$ will be $\tfrac{y'}{y_{n+1}}$ and is suppressed. 
    We can differentiate in the prime and boundary directions, for $i,j,k,\ell \in \{1,\ldots, n\}$,
    \begin{align*}
        \varphi_i &= (1 + \gamma)\left(y_{n+1}v\right)^\gamma v_i,\\
        \varphi_{(n+1)} &= (1 + \gamma)\left(y_{n+1}v\right)^\gamma\left(v -y_{n+1}\sum_k v_k\left(\frac{y'}{y_{n+1}}\right)\frac{y_k}{y_{n+1}^2}\right)= (1+\gamma)\frac{\varphi}{y_{n+1}} - \sum_k \varphi_k \frac{y_k}{y_{n+1}},
    \end{align*}
    and the second derivatives are
    \begin{align*}
        \varphi_{ij}=& (1 + \gamma)\gamma\left(y_{n+1}v\right)^{\gamma - 1}v_iv_j +(1 +\gamma)\left(y_{n+1}v\right)^\gamma \frac{v_{ij}}{y_{n+1}}\\
        =&(1 + \gamma)\left(y_{n+1}v\right)^{\gamma-1}( vv_{ij} + \gamma v_iv_j),\\
        \varphi_{i(n+1)} =& (1 + \gamma)\gamma \left(y_{n+1}v\right)^{\gamma-1}\left(v - \sum_k v_k\frac{y_k}{y_{n+1}}\right)v_i - (1 + \gamma)\left(y_{n+1}v\right)^\gamma \sum_k\frac{v_{ik}y_k}{y_{n+1}^2},\\
        \varphi_{(n+1)(n+1)}=& (1 + \gamma)\gamma \left(y_{n+1}v\right)^{\gamma - 1}\left(v - \sum_k v_k \frac{y_k}{y_{n+1}}\right)^2\\
        &+(1 +\gamma)\left(y_{n+1} v\right)^\gamma \left(-\sum v_k \frac{y_k}{y_{n+1}^2} + \sum_k v_k \frac{y_k}{y_{n+1}^2} + \sum_{k,j} v_{kj}\frac{y_ky_j}{y_{n+1}^3}\right)\\
        =&  (1 + \gamma)\gamma \left(y_{n+1}v\right)^{\gamma - 1}\left(v - \sum_k v_k \frac{y_k}{y_{n+1}}\right)^2 +(1 +\gamma)\left(y_{n+1} v\right)^\gamma \left(\sum_{k,\ell} v_{k \ell}\frac{y_ky_\ell}{y_{n+1}^3}\right).\\
    \end{align*}
    We can express
    \[
    \det \varphi_{ij} = (1+\gamma)^n\left(y_{n+1} v\right)^{n(\gamma - 1)} v^n\left(1 + \frac{\gamma v^{ij}v_i v_j}{v}\right)\det v_{ij},
    \]
    and $M_{ij} = \varphi_{(n+1)(n+1)}-\varphi^{ij}\varphi_{i(n+1)}\varphi_{j(n+1)}$ can be computed
    \begin{align*}
M_{ij}&=(1+\gamma)\left(y_{n+1}v\right)^{\gamma-1}\Bigg[
      \gamma\left(v-\frac{y_k v_k}{y_{n+1}}\right)^2
      +v\frac{v_{k\ell}y_k y_\ell}{y_{n+1}^{2}}
      \\
      &\quad-\Bigl(\gamma \left(v-\frac{y_r v_r}{y_{n+1}}\right)v_i
             -\frac{v}{y_{n+1}}v_{ip}y_p\Bigr)
        \left(v v_{ij}+\gamma v_i v_j\right)^{-1}
        \Bigl(\gamma \left(v-\frac{y_s v_s}{y_{n+1}}\right)v_j
             -\frac{v}{y_{n+1}}v_{jq}y_q\Bigr)
    \Bigg] \\
&=(1+\gamma)\left(y_{n+1}v\right)^{\gamma-1}
    \left[\frac{\gamma v^{2}}
                {1+\gamma v^{-1}v^{ij}v_i v_j}\right].
\end{align*}
    We can use the Laplace expansion combined with the above equations to compute
    \begin{align*}
        \det D^2 \varphi &= (\det \varphi_{ij})(\varphi_{(n+1)(n+1)} - \varphi^{ij}\varphi_{i(n+1)}\varphi_{j(n+1)})\\
        &= (1 + \gamma)^{n+1} (y_{n+1}v)^{(n+1)(\gamma - 1)}v^n \left(1 + \frac{\gamma v^{ij}v_iv_j}{v}\right)\left[\frac{\gamma v^2}{1 + \gamma v^{-1}v^{ij}v_iv_j}\right] \det v_{ij} \\
        &=  (1 + \gamma)^{n+1}\gamma \varphi^{\frac{(n+1)(\gamma - 1)}{1 + \gamma}}v^{n+2} \det v_{ij}.
    \end{align*}
    Recalling the definition $-v^\star = v - \sum_k y_k v_k$, we can simplify $\varphi_{(n+1)}$ as
    \begin{align*}
    \varphi_{(n+1)} &= \frac{1}{y_{n+1}}((1 + \gamma) \varphi - \sum_k \varphi_k y_k) \\
    &= \frac{1}{y_{n+1}}\left((1 + \gamma)\left(y_{n+1}v\right)^{1+\gamma} - \sum_k (1 + \gamma) \left(y_{n+1}v\right)^\gamma v_ky_k^2 \right)\\
    &= (1 + \gamma)y_{n+1}^{\gamma }v^\gamma (-v^\star)\\
    &= (1 +\gamma)\varphi^\frac{\gamma}{1+ \gamma}(-v^\star),
    \end{align*}
    showing
    \[
    (-v^\star)^{-\beta} = (1 + \gamma)^\beta\varphi_{n+1}^{-\beta}\varphi^{\frac{\beta\gamma}{1 + \gamma}}.
    \]
    We can therefore compute
    \[
    \det v_{ij} = h(y') v^{-(n+2+\alpha)}(-v^\star)^\beta  = (1 + \gamma)^\beta h v^{-(n+2 + \alpha)}\varphi_{n+1}^{-\beta}\varphi^{\frac{\beta\gamma}{1 + \gamma}},
    \]
    which combined with the computation of $\det D^2 \varphi$ shows
    \begin{align*}
    \det D^2 \varphi &= (1 + \gamma)^{n+1 + \beta}\gamma \varphi^{\frac{(n+1)(\gamma -1) + \beta \gamma}{1 + \gamma}}h v^{-\alpha}\varphi_{n+1}^{-\beta}\\
    &= (1 + \gamma)^{n+1 + \beta}\gamma \varphi^{\frac{(n+1)(\gamma -1) + \beta \gamma - \alpha}{1 + \gamma}}\rho^{\alpha}\varphi_{n+1}^{-\beta}
    \end{align*}
    using $y_{n+1} = \varphi^\frac{1}{1 + \gamma}v^{-1}$. 
    Lastly, the exponent $\frac{(n+1)(\gamma - 1) + \beta\gamma -\alpha}{1 + \gamma}$ vanishes by construction of $\gamma = \frac{n + 1 + \alpha}{n + 1 + \beta}$, so we conclude that 
    \[
    \varphi_{n+1}^\beta \det D^2 \varphi = C \rho^{\alpha} ,
    \]
    showing $\varphi$ solves equation~\eqref{eq: OTCONES}, after rescaling so that $C=1$.
\end{proof}

We can now prove Theorem~\ref{thm: introOT}.

\begin{proof}[Proof of Theorem~\ref{thm: introOT}]
By Lemma~\ref{lem: homogReductOT} it suffices to solve~\eqref{eqn:simplexLinkEqn}. 
To do this we try to apply Theorem~\ref{thm: mainH23}. The relevant pair $(F,G)$ is
\[
F(s) = \frac{1}{\beta+1}s^{1+\beta}, \qquad G(s) = \frac{-1}{n+1+\alpha}s^{-(n+1+\alpha)},\qquad h=h(y').
\]
Property~\refsHc\ is satisfied provided $\beta >\alpha$. 
Since $h\sim d(y', \del P)^{\alpha}$, $\mathtt{v}_o=\alpha$ is a vanishing order of $h$, (see Remark~\ref{rk: vanishOrder}), and so the result follows from Theorem~\ref{thm: mainH23}. 
\end{proof}

\subsubsection{Toric K\"ahler-Einstein metrics}
We end with a final example of geometric interest. Let $P\subset \mathbb{R}^n$ be a bounded convex polytope with barycenter at the origin. 
Then, there is a convex function $u:\mathbb{R}^n \rightarrow \mathbb{R}$, unique up to translation, such that
\[
\begin{aligned}
    \det D^2u &= e^{-u}\; \text{ in }\mathbb{R}^n,\\
    \nabla u(\mathbb{R}^n) &=P.
\end{aligned}
\]
This theorem was first established by Wang-Zhu \cite{Wang-Zhu} for reflexive Delzant polytopes, and for general polytopes by Berman-Berndtsson \cite{Berman-Berndtsson}. 
Geometrically, $u$ gives rise to a toric K\"ahler-Einstein metric on the toric log Fano variety corresponding to $P$. We refer the reader to \cite{Berman-Berndtsson} for a discussion of this correspondence.

Motivated by these results, we consider the case $F(s)= e^{s}-1$, and $G(s)=s$. 
The pair $(F,G)$ satisfies property \refHa. 
Let $P$ be a compact, convex polytope with non-empty interior and barycenter at the origin. 
By applying Theorem~\ref{thm:Barycenter} and Corollary~\ref{cor: uniquenessForExamples}, we deduce that, for all $\Lambda$ sufficiently large, there exists a function $u_{\Lambda}:\mathbb{R}^n \rightarrow \mathbb{R}$, unique up to translation, such that $\{ u_{\Lambda}<0\}$ has barycenter at the origin, and
\[
\begin{aligned}
    \det D^2u_{\Lambda} &= C_{\Lambda} e^{-u_{\Lambda{}}}\chi_{\{u_{\Lambda}<0\}} \; \text{ in } \mathbb{R}^n,\\
    \nabla u_{\Lambda}(\mathbb{R}^n) &= P.
\end{aligned}
\]
Such a function defines a free boundary K\"ahler-Einstein metric on the toric log-Fano variety $X$ corresponding to $P$ \cite{Berman-Berndtsson}. Assume that $X$ is a smooth Fano variety, and let $\sigma$ denote the unique, torus-invariant element of $H^{0}(X,-K_{X})$. 
Then, $e^{-u}$ defines a metric $h_u$ on $-K_{X}$ and we obtain a (singular) K\"ahler metric $\omega= - \ddbar \log h_u$ on $X$ satisfying
\[
\begin{aligned}
{\rm Ric}(\omega) = \omega &\quad \text{ in } \quad  \{ |\sigma|_{h_{u}} > 1\} \subsetneq X,\\
\omega^n =0 & \quad \text{ in }\quad   \{ |\sigma|_{h_{u}} < 1\} \ne \emptyset.
\end{aligned}
\]
Our results also apply more generally to yield the existence of free boundary K\"ahler-Ricci solitons \cite{Wang-Zhu, Berman-Berndtsson} and toric free boundary Mabuchi solitons \cite{Yao}.

\subsection{Non-examples, and further questions}

In this section we discuss the extent to which our assumptions can be weakened, and pose some natural questions.

\subsubsection{Non-examples} 
First we remark that the existence of critical points of $\mathcal{E}_{\Lambda}$ for all $\Lambda>0$ is clearly false; this can easily be deduced from Lemma~\ref{lem: exponentialJFix}.

As a more interesting non-example, consider the case $F(s)=s$ and $G(s)=-s^{-(n+1)}$. 
This equation is motivated by two considerations. 
First, the pair $(F,G)$ does not satisfy Property~\refHc, since $s^{n}F(s)G(s) \rightarrow -1$ as $s\rightarrow 0$. 
Secondly, this equation is related to the existence of elliptic affine hemispheres, by Klartag \cite{Klartag}, and to the existence of optimal transport maps between strict cones and half-spaces equipped with the uniform density \cite{TristanFreidYau}. 
Let $P$ be a compact, convex set. 
Suppose that $u$ solves the free boundary problem
\begin{equation}\label{eq: badEquation}
\begin{aligned}
\det D^{2}u &= \lambda (u^*)^{n+2}\chi_{u<0}\\
\nabla u(\mathbb{R}^n) &= P.
\end{aligned}
\end{equation}
Consider the radial graph given by
\[
R(u) = \left\{-\frac{1}{u(x)}(x,1) : x \in \{u<0\}\right\}.
\]
If $u$ is convex, then $R(u)$ is the graph of a convex function $\phi$ defined by
\[
\phi\left(\frac{x}{-u(x)}\right) = \frac{1}{-u(x)}
\]
and $\phi$ solves \cite[Section 5]{Tong-Yau}
\[
\begin{aligned}
\phi^{n+2}\det D^{2}\phi&=  \lambda \; \text{ in } \mathbb{R}^n\\
\nabla \phi(\mathbb{R}^n) &= P^\circ.
\end{aligned}
\]
Then the graph of $\psi = \phi^*$ defines an affine hemisphere. For example, if $P = B_1$, and if $u=u(r)$ is radially symmetric, then~\eqref{eq: badEquation} becomes the ODE
\begin{equation}\label{eq: ODEhemisphere}
    u_{rr} \left(\frac{u_r}{r}\right)^{n-1} = (r u_r - u)^{n+2}.
\end{equation}
Any radially symmetric solution to~\eqref{eq: badEquation} must have $u'(0)=0$, and $u(0)<0$. Observe that $u(r)= -\sqrt{1 - r^2}$ is a solution corresponding to the round hemisphere, and by uniqueness, any radially symmetric solution of~\eqref{eq: badEquation} satisfying $u'(0)=0$ and $u(0)<0$ must be of the form $-a\sqrt{1-r^2}$ for some $a>0$. 
Now, $u(1)=0$, but $u'(1)=\infty$.
This strongly suggests that a solution to~\eqref{eq: badEquation} does not exist.
Further evidence can be obtained by considering the Legendre dual problem
\begin{equation}\label{eq: badEquationLegendre}
\begin{aligned}
\det D^{2}v &= \lambda v^{-(n+2)} \; \text{ in } B_1\\
v^\star\big|_{\del B_1} &= 0.
\end{aligned}
\end{equation}
By reducing to an ODE for radially symmetric $v=v(r)$, and imposing $v'(0)=0$, one can solve the equation numerically and see that $v(0)\rightarrow 0$, upon attempting to impose the boundary condition $v^\star\big|_{\del B_1} = 0$. 
Note that when $n=1$, radial symmetry is automatic, and non-solvability of~\eqref{eq: badEquationLegendre} was observed in \cite[Remark 3.8]{Collins-Li}. This discussion suggests that the assumptions of~\refHc\, cannot be substantially weakened.

\subsubsection{Possible extensions}

On the other hand, there seem to be interesting extensions of~\refHb\, under which a solution may exist. 
For example, if $F(s) = s^2$, and $G(s) = -s^{-(n+5)}+s$, then Property~\refHb\, is not satisfied, but numerical calculations suggest that a solution still exists. 
For this choice of pair $(F,G)$, the energy lower bounds following from Proposition~\ref{prop: G-bound} yield a function $\hat{\cE}(x,D)$, for which the sublevel sets $\{(x,D) : \hat{\cE}(x,D) \leq A\}$ are not compact, but instead have a compact component and a non-compact component; see Figure~\ref{fig:compactnessEx2Components} below.
\begin{figure}[H]
    \centering
    \includegraphics[width=0.5\linewidth]{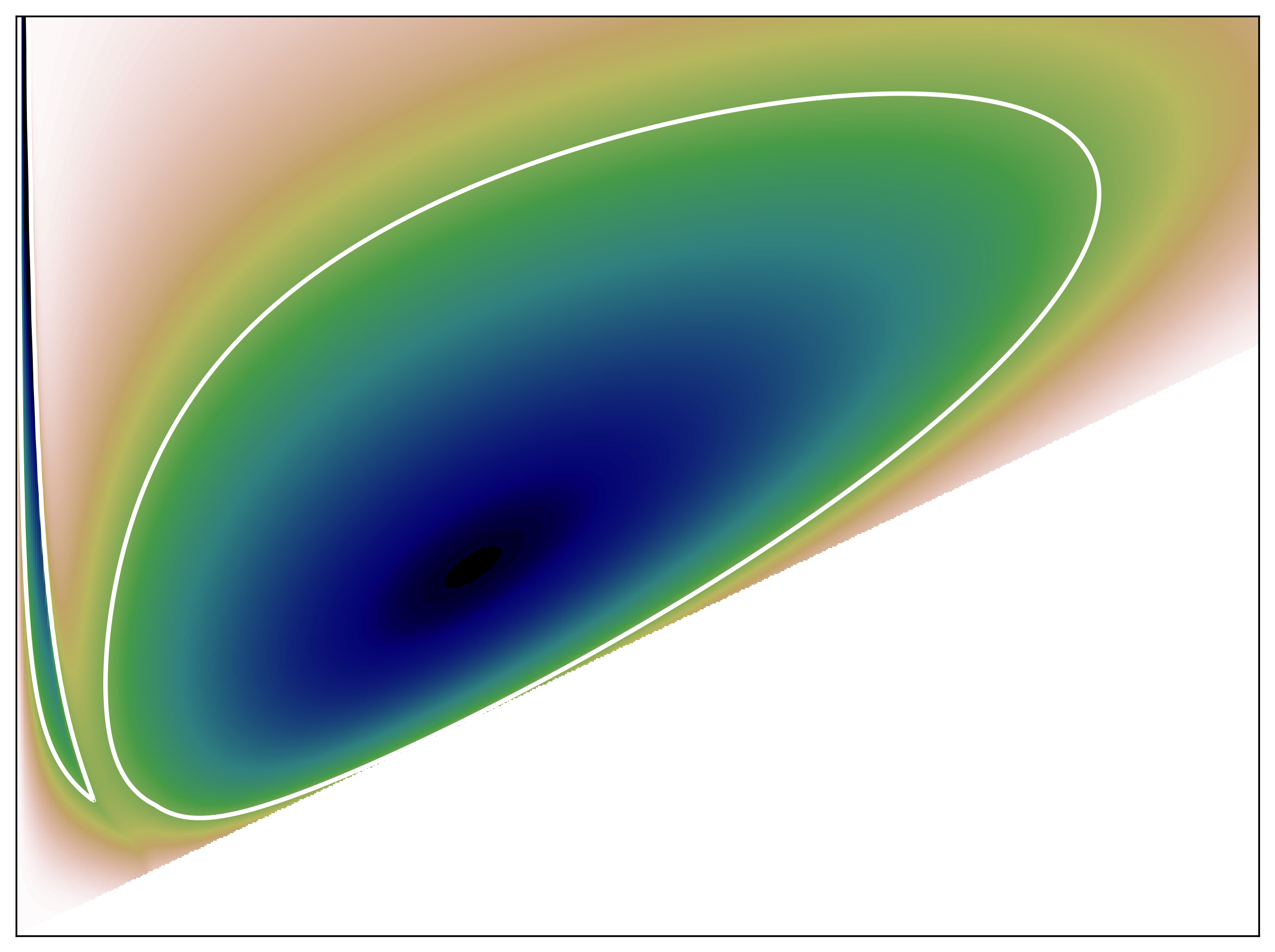}
    \caption{The regions enclosed by the white boundary form the sublevel set of $\hat{\cE}(x,D)$ for the pair $F(s) = \tfrac{1}{2}s^2$, and $G(s) = -s^{-(n+5)}+s$ for $n = 4$. 
    The sublevel set has both a compact part, where we expect a solution to exist, and a non-compact part.}
    \label{fig:compactnessEx2Components}
\end{figure}

It is possible that such solutions may be rigorously proved to exist by developing a suitably ``localized" version of our main result.

\medskip 

\end{document}

%% file: derivatives.tex
\newcommand{\nc}{\newcommand}
\nc{\p}{\partial}
\nc{\pb}{\partial_b}
\nc{\pc}{\partial_c}
\nc{\pd}{\partial_d}
\nc{\pe}{\partial_e}
\nc{\pf}{\partial_f}
\nc{\pg}{\partial_g}
\nc{\ph}{\partial_h}
\nc{\pari}{\partial_i}
\nc{\pj}{\partial_j}
\nc{\pk}{\partial_k}
\nc{\pl}{\partial_l}
\nc{\pell}{\partial_\ell}
\nc{\parm}{\partial_m}
\nc{\pn}{\partial_n}
\nc{\po}{\partial_o}
\nc{\pp}{\partial_p}
\nc{\pq}{\partial_q}
\nc{\pr}{\partial_r}
\nc{\ps}{\partial_s}
\nc{\pt}{\partial_t}
\nc{\pu}{\partial_u}
\nc{\pv}{\partial_v}
\nc{\pw}{\partial_w}
\nc{\px}{\partial_x}
\nc{\py}{\partial_y}
\nc{\pz}{\partial_z}

\nc{\pabar}{\partial_{\ol{a}}}
\nc{\pbbar}{\partial_{\ol{b}}}
\nc{\pcbar}{\partial_{\ol{c}}}
\nc{\pdbar}{\partial_{\ol{d}}}
\nc{\pebar}{\partial_{\ol{e}}}
\nc{\pfbar}{\partial_{\ol{f}}}
\nc{\pgbar}{\partial_{\ol{g}}}
\nc{\phbar}{\partial_{\ol{h}}}
\nc{\pibar}{\partial_{\ol{i}}}
\nc{\pjbar}{\partial_{\ol{j}}}
\nc{\pkbar}{\partial_{\ol{k}}}
\nc{\plbar}{\partial_{\ol{l}}}
\nc{\pellbar}{\partial_{\ol{\ell}}}
\nc{\pmbar}{\partial_{\ol{m}}}
\nc{\pnbar}{\partial_{\ol{n}}}
\nc{\pobar}{\partial_{\ol{o}}}
\nc{\ppbar}{\partial_{\ol{p}}}
\nc{\pqbar}{\partial_{\ol{q}}}
\nc{\prbar}{\partial_{\ol{r}}}
\nc{\psbar}{\partial_{\ol{s}}}
\nc{\ptbar}{\partial_{\ol{t}}}
\nc{\pubar}{\partial_{\ol{u}}}
\nc{\pvbar}{\partial_{\ol{v}}}
\nc{\pwbar}{\partial_{\ol{w}}}
\nc{\pxbar}{\partial_{\ol{x}}}
\nc{\pybar}{\partial_{\ol{y}}}
\nc{\pzbar}{\partial_{\ol{z}}}

\nc{\nababar}{\nabla_{\ol{a}}}
\nc{\nabbbar}{\nabla_{\ol{b}}}
\nc{\nabcbar}{\nabla_{\ol{c}}}
\nc{\nabdbar}{\nabla_{\ol{d}}}
\nc{\nabebar}{\nabla_{\ol{e}}}
\nc{\nabfbar}{\nabla_{\ol{f}}}
\nc{\nabgbar}{\nabla_{\ol{g}}}
\nc{\nabhbar}{\nabla_{\ol{h}}}
\nc{\nabibar}{\nabla_{\ol{i}}}
\nc{\nabjbar}{\nabla_{\ol{j}}}
\nc{\nabkbar}{\nabla_{\ol{k}}}
\nc{\nablbar}{\nabla_{\ol{l}}}
\nc{\nabelllbar}{\nabla_{\ol{\ell}}}
\nc{\nabmbar}{\nabla_{\ol{m}}}
\nc{\nabnbar}{\nabla_{\ol{n}}}
\nc{\nabobar}{\nabla_{\ol{o}}}
\nc{\nabpbar}{\nabla_{\ol{p}}}
\nc{\nabqbar}{\nabla_{\ol{q}}}
\nc{\nabrbar}{\nabla_{\ol{r}}}
\nc{\nabsbar}{\nabla_{\ol{s}}}
\nc{\nabtbar}{\nabla_{\ol{t}}}
\nc{\nabubar}{\nabla_{\ol{u}}}
\nc{\nabvbar}{\nabla_{\ol{v}}}
\nc{\nabwbar}{\nabla_{\ol{w}}}
\nc{\nabxbar}{\nabla_{\ol{x}}}
\nc{\nabybar}{\nabla_{\ol{y}}}
\nc{\nabzbar}{\nabla_{\ol{z}}}

\nc{\naba}{\nabla_{a}}
\nc{\nabb}{\nabla_{b}}
\nc{\nabc}{\nabla_{c}}
\nc{\nabd}{\nabla_{d}}
\nc{\nabe}{\nabla_{e}}
\nc{\nabf}{\nabla_{f}}
\nc{\nabg}{\nabla_{g}}
\nc{\nabh}{\nabla_{h}}
\nc{\nabi}{\nabla_{i}}
\nc{\nabj}{\nabla_{j}}
\nc{\nabk}{\nabla_{k}}
\nc{\nabl}{\nabla_{l}}
\nc{\nabm}{\nabla_{m}}
\nc{\nabn}{\nabla_{n}}
\nc{\nabo}{\nabla_{o}}
\nc{\nabp}{\nabla_{p}}
\nc{\nabq}{\nabla_{q}}
\nc{\nabr}{\nabla_{r}}
\nc{\nabs}{\nabla_{s}}
\nc{\nabt}{\nabla_{t}}
\nc{\nabu}{\nabla_{u}}
\nc{\nabv}{\nabla_{v}}
\nc{\nabw}{\nabla_{w}}
\nc{\nabx}{\nabla_{x}}
\nc{\naby}{\nabla_{y}}
\nc{\nabz}{\nabla_{z}}

\nc{\ola}{\ol{a}}
\nc{\olb}{\ol{b}}
\nc{\olc}{\ol{c}}
\nc{\old}{\ol{d}}
\nc{\ole}{\ol{e}}
\nc{\olf}{\ol{f}}
\nc{\olg}{\ol{g}}
\nc{\olh}{\ol{h}}
\nc{\oli}{\ol{i}}
\nc{\olj}{\ol{j}}
\nc{\olk}{\ol{k}}
\nc{\oll}{\ol{l}}
\nc{\olm}{\ol{m}}
\nc{\oln}{\ol{n}}
\nc{\olo}{\ol{o}}
\nc{\olp}{\ol{p}}
\nc{\olq}{\ol{q}}
\nc{\olr}{\ol{r}}
\nc{\ols}{\ol{s}}
\nc{\olt}{\ol{t}}
\nc{\olu}{\ol{u}}
\nc{\olv}{\ol{v}}
\nc{\olw}{\ol{w}}
\nc{\olx}{\ol{x}}
\nc{\oly}{\ol{y}}
\nc{\olz}{\ol{z}}

%% file: benjymacros.tex
\newcommand{\la}{\langle}
\newcommand{\rg}{\rangle}

\newcommand{\mr}[1]{{\rm #1}}

\usepackage{wrapfig,caption}
\newcommand{\cC}{\mathcal{C}}
\newcommand{\cE}{\mathcal{E}}

\newcommand{\bR}{\mathbb{R}}
\newcommand{\bS}{\mathbb{S}}

\nc{\sn}{\mr{sn}}
\nc{\cn}{\mr{cn}}
\nc{\dn}{\mr{dn}}

\nc{\ol}{\overline}
\nc{\ul}{\underline}
\nc{\iddbar}{\sqrt{-1}\partial \overline{\partial}}

\newcommand{\RNum}[1]{\uppercase\expandafter{\romannumeral #1\relax}}
\nc{\todo}{{{\color{red}\huge {TO DO}}\errormessage} }
\nc{\lob}{\Lambda}

\nc{\LAMBDA}{\lambda}